\documentclass[11pt]{amsart}

\usepackage[T1]{fontenc}
\usepackage{lmodern}
\usepackage[a4paper,margin=29mm]{geometry}
\usepackage{microtype}
\usepackage{amsmath,amssymb,mathtools}
\usepackage{enumitem}
\usepackage{xcolor}
\usepackage{tikz}
\usetikzlibrary{arrows.meta}
\usepackage[colorlinks=true,linkcolor=blue!55!black,citecolor=blue!55!black,urlcolor=blue!65!black]{hyperref}

\newtheorem{theorem}{Theorem}[section]
\newtheorem{proposition}[theorem]{Proposition}
\newtheorem{lemma}[theorem]{Lemma}

\theoremstyle{definition}

\theoremstyle{remark}

\numberwithin{equation}{section}
\setlist[enumerate]{leftmargin=*,label=(\roman*)}

\newcommand{\R}{\mathbb R}
\newcommand{\dd}{\,\mathrm d}
\newcommand{\ip}[1]{\langle #1\rangle}
\newcommand{\FL}{\mathcal F L}
\newcommand{\B}{\mathcal B}
\newcommand{\supp}{\operatorname{supp}}
\newcommand{\res}{\mathcal R}

\title[Sharp Barron regularity for Coulombic wave functions]{Sharp Barron Regularity Results for Coulombic Many-Electron Wave Functions}
\hypersetup{
  pdftitle={Sharp Barron Regularity Results for Coulombic Many-Electron Wave Functions}, pdfauthor={Pingbing Ming; Hao Yu}, pdfsubject={Sharp spectral Barron regularity for Coulombic many-electron wave functions}, pdfkeywords={Coulomb Hamiltonian, electronic wave function, Jastrow factor, spectral Barron space, Fourier--Lebesgue space} }

\author[P. Ming]{Pingbing Ming}
\author[H. Yu]{Hao Yu}
\address{SKLMS, Institute of Computational Mathematics and Scientific/Engineering Computing, Academy of Mathematics and Systems Science, Chinese Academy of Sciences,
Beijing 100190, China; School of Mathematical Sciences, University of Chinese Academy of Sciences, Beijing 100049, China}
\email[Pingbing Ming]{mpb@lsec.cc.ac.cn}
\email[Hao Yu]{yuhao@amss.ac.cn}


\begin{document}

\begin{abstract}
  We establish sharp Barron regularity for Coulombic many-electron wave functions after extraction of the universal cut-off Jastrow factors. Following the factorization of Fournais et al.~\cite[Definition~1.4]{FournaisEtAl2005}, for a Coulombic eigenfunction $\psi$ we define the successive quotients by $\phi=e^{-F_{2,\mathrm{cut}}}\psi$ and $\phi_3=e^{-F_{3,\mathrm{cut}}}\phi =e^{-(F_{2,\mathrm{cut}}+F_{3,\mathrm{cut}})}\psi.$
  Then
  \[ \phi,\phi_3\in\B^s(\R^{3N}) \qquad\text{for every }s<2. \]
  This range is optimal among universal factorizations. No factor depending only on the particle number and the nuclear data, but not on the eigenfunction or its eigenvalue, can make every corresponding quotient belong to $\B^2$.

  We also determine the exact endpoint growth. Writing $\varepsilon=2-s$, we prove that, for either $u=\phi$ or $u=\phi_3$, there is a computable constant $M$ independent of $\varepsilon$ such that
  \[ 
  \|u\|_{\B^{2-\varepsilon}} \leq \frac{M}{\varepsilon^{2}} \|u\|_{\B^1}. 
  \]
  For the unperturbed two-electron atom we prove, with a constant independent of $\varepsilon$,
  \[ \left| \|\phi_3\|_{\B^{2-\varepsilon}} - \frac{32\pi Z|\phi_3(0,0)|}{\varepsilon^{2}} \right| \leq \frac{C}{\varepsilon}. \]
  Hence the quadratic rate in the upper bound is sharp whenever \(\lvert\phi_3(0,0)\rvert\neq0\), as is the case for the ground state.
\end{abstract}

\maketitle

\section{Introduction and main results}\label{sec:introduction}

We study the electronic Schr\"odinger equation for integers $N,L\geq1$, where $N$ is the number of electrons and $L$ is the number of fixed nuclei. Let $x=(x_1,\ldots,x_N)\in\R^{3N}$, and let $R_1,\ldots,R_L\in\R^3$ be distinct nuclear positions with charges $Z_1,\ldots,Z_L>0$. We use the convention in which the kinetic energy is $-\sum_{i=1}^N\Delta_{x_i}$. The electronic Hamiltonian is
\begin{equation}\label{eq:H} H=-\sum_{i=1}^N\Delta_{x_i}+V, \qquad V=-\sum_{i=1}^N\sum_{\nu=1}^L\frac{Z_\nu}{|x_i-R_\nu|} +\sum_{1\leq i<j\leq N}\frac1{|x_i-x_j|}. \end{equation}
The internuclear repulsion is constant in the fixed-nuclei approximation and is absorbed into the eigenvalue.  The standard self-adjoint realization of $H$ on $L^2(\R^{3N})$ has domain $H^2(\R^{3N})$.  We consider nonzero eigenfunctions satisfying
\[ H\psi=E\psi, \qquad E\in\R, \qquad 0\ne\psi\in H^2(\R^{3N}). \]
The Hamiltonian \eqref{eq:H} is the standard nonrelativistic fixed-nuclei model for atoms and molecules \cite{FournaisEtAl2005,Yserentant2010}.  Its eigenfunctions depend on $3N$ spatial variables, and their regularity directly affects the complexity of their high-dimensional approximation \cite{Yserentant2010}.

Away from the electron--nucleus and electron--electron collision sets, the Coulomb potential is real analytic, and elliptic analytic regularity gives \cite{FournaisEtAl2009}
\[ \psi\in C^\omega\!\left( \R^{3N}\setminus \left[ \bigcup_{i=1}^N\bigcup_{\nu=1}^L\{x_i=R_\nu\} \bigcup \bigcup_{1\leq i<j\leq N}\{x_i=x_j\} \right] \right). \]
Kato~\cite{Kato1957} proved that Coulombic eigenfunctions are locally Lipschitz and derived spherical-average cusp conditions at simple two-particle coalescences. Fournais et al.~\cite{FournaisEtAl2005} subsequently constructed explicit two- and three-particle factors $F_2$ and $F_3$, independent of $E$ and of the particular eigenfunction, such that
\[ e^{-(F_2+F_3)}\psi\in C_{\mathrm{loc}}^{1,1}(\R^{3N}). \]
The function $F_2$ collects the pairwise Coulomb cusp profiles, whereas $F_3$ is an explicit logarithmic correction involving two electrons and one nucleus. They also proved that $C_{\mathrm{loc}}^{1,1}$ is optimal among factorizations by a universal factor depending only on the particle number and the nuclear data.
Taken together, these results show that extracting the universal Jastrow factors raises the local H\"older regularity from $C_{\mathrm{loc}}^{0,1}$ to $C_{\mathrm{loc}}^{1,1}$, a gain of one full order. The present paper establishes the analogous gain in the global spectral Barron scale. After extracting the cut-off Jastrow factors, the regularity range improves from $s<1$ to $s<2$, and we determine the precise norm growth as $s\uparrow2$.

For $\varphi\in\mathcal S(\R^d)$, we use the unitary Fourier transform
\[ \widehat\varphi(\xi)=(2\pi)^{-d/2} \int_{\R^d}e^{-ix\cdot\xi}\varphi(x)\dd x \]
and extend it to $\mathcal S'(\R^d)$ by duality.  For $s\in\R$, the spectral Barron space is
\[ \B^s(\R^d) :=\left\{ u\in\mathcal S'(\R^d): \ip{\cdot}^{s}\widehat u\in L^1(\R^d) \right\}, \]
equipped with the norm
\[ \|u\|_{\B^s(\R^d)} :=\int_{\R^d}\ip{\xi}^{s}|\widehat u(\xi)|\dd\xi, \qquad \ip{\xi}=(1+|\xi|^2)^{1/2}. \]
For the Coulombic wave function, Yserentant \cite[Theorem~4.7]{Yserentant2026} proved that every eigenfunction with eigenvalue below the essential spectrum belongs to $\B^s(\R^{3N})$ for $0\leq s<1$, with norm bounded by $\mathcal{O}(1-s)^{-1}$.  The hydrogenic ground state does not belong to $\B^1(\R^3)$ and has precisely this simple-pole growth.  Our earlier work~\cite[Theorem~2.4, Corollary~2.7, and Example~2.8]{MingYu2025} treats a broader class of one- and two-particle potentials under low-rank Fourier--Lebesgue assumptions.  In the three-dimensional Coulomb case it recovers the same range $s<1$ and the same $\mathcal{O}((1-s)^{-1})$ bound, while the radial example shows that this range cannot in general be extended to $s=1$.
\subsection{Cut-off Jastrow factors}

We first define a cut-off version of the universal pairwise cusp factor and then add the universal three-particle logarithmic correction.

Fix $\chi\in C_c^\infty([0,\infty))$ such that $0\leq\chi\leq1$, $\chi=1$ on $[0,1]$, and $\chi=0$ on $[2,\infty)$.  Define
\begin{align*}
  F_2(x)
   & =-\frac12\sum_{i=1}^N\sum_{\nu=1}^L Z_\nu|x_i-R_\nu|
   +\frac14\sum_{1\leq i<j\leq N}|x_i-x_j|,                                      \\
  F_{2,\mathrm{cut}}(x)
   & =-\frac12\sum_{i=1}^N\sum_{\nu=1}^L Z_\nu\chi(|x_i-R_\nu|)|x_i-R_\nu|                       \\
   & \quad+\frac14\sum_{1\leq i<j\leq N}\chi(|x_i-x_j|)|x_i-x_j|.
\end{align*}
Then $\Delta F_2=V$ in the sense of distribution. This identity is the basic cancellation mechanism.

For $(x,y)\in\R^6$, set
\begin{equation}\label{eq:q} q(x,y)=\chi(|x|)\chi(|y|)(x\cdot y) \log(|x|^2+|y|^2), \end{equation}
where the value at $(0,0)$ is defined by continuity. The explicit cut-off three-particle factor is
\begin{equation*} F_{3,\mathrm{cut}}(x) =\frac{2-\pi}{12\pi} \sum_{\nu=1}^L\sum_{1\leq i<j\leq N}Z_\nu q(x_i-R_\nu,x_j-R_\nu). \end{equation*}
It removes the universal isotropic three-particle logarithm. These are precisely the cut-off factors introduced in~\cite{FournaisEtAl2005}.

For the eigenfunction $\psi$ fixed above, define the two quotient functions successively by
\begin{equation}\label{eq:quotients} \phi=e^{-F_{2,\mathrm{cut}}}\psi, \qquad \phi_3=e^{-F_{3,\mathrm{cut}}}\phi. \end{equation}
A combination of the two definitions in~\eqref{eq:quotients} gives $\phi_3=e^{-(F_{2,\mathrm{cut}}+F_{3,\mathrm{cut}})}\psi$.
Thus $\phi$ removes only the universal two-particle cusp factor, whereas $\phi_3$ removes, in addition, the explicit cut-off three-particle logarithmic factor.

\subsection{Main results}

The following theorem states the global regularity range, the norm growth as $s\uparrow2$, and the sharpness of both conclusions.

\begin{theorem}[Main results]\label{thm:main}
  Let $\psi\in H^2(\R^{3N})\setminus\{0\}$ satisfy $H\psi=E\psi$, and let $\phi$ and $\phi_3$ be defined by \eqref{eq:quotients}.
  \begin{enumerate}
    \item For every $s<2$,
          \begin{equation}\label{eq:regularity-range} \phi,\phi_3\in\B^s(\R^{3N}). \end{equation}
    \item For either $u=\phi$ or $u=\phi_3$, there is a computable constant $M\geq0$, depending only on the choice of $u$, $N$, $E$, the nuclear charges $\{Z_\nu\}$, the relative nuclear positions $\{R_\mu-R_\nu\}$, and the fixed cut-off profile $\chi$, such that
          \begin{equation}\label{eq:upper} \|u\|_{\B^{2-\varepsilon}(\R^{3N})} \leq \frac{M}{\varepsilon^2} \|u\|_{\B^1(\R^{3N})}, \qquad 0<\varepsilon\leq\frac12. \end{equation}
          Moreover, $\|u\|_{\B^1(\R^{3N})}\leq C\|u\|_{H^1(\R^{3N})}$ with $C$ independent of $\varepsilon$.
    \item The range $s<2$ in \eqref{eq:regularity-range} cannot be extended to $s=2$ by any universal factorization\footnote{This is the class fixed in the setup of Fournais et al.~\cite{FournaisEtAl2005}. An admissible universal factor is a finite, strictly positive, locally Lipschitz function $G_{N,R,Z}$; the same factor is used for every eigenfunction and eigenvalue of the corresponding Coulomb Hamiltonian, and the factorization is understood almost everywhere.} of the following form. There is no multiplicative factor $G$, depending only on the particle number and the nuclear positions and charges, but not on $\psi$ or $E$, such that every Coulombic eigenfunction can be written as
          \[ \psi=Gu, \qquad u\in\B^2(\R^{3N}). \]
    \item Let $N=2$, $L=1$, $R_1=0$, and $Z_1=Z>0$. Then \eqref{eq:H} becomes the unperturbed two-electron atomic Hamiltonian
          \[ H_Z=-\Delta_x-\Delta_y-\frac Z{|x|}-\frac Z{|y|}+\frac1{|x-y|} \quad\text{on }\R^6. \]
          If $H_Z\psi=E\psi$ with $E<\inf\sigma_{\mathrm{ess}}(H_Z)$, and if $\phi_3$ is the corresponding quotient in \eqref{eq:quotients}, then there is $C_{\rm sharp}<\infty$ such that
          \begin{equation}\label{eq:sharp-main} \left|\|\phi_3\|_{\B^{2-\varepsilon}(\R^6)} -\frac{32\pi Z|\phi_3(0,0)|}{\varepsilon^2}\right| \leq\frac{C_{\rm sharp}}{\varepsilon}, \end{equation}
          for $0<\varepsilon\leq1/2$.
          Consequently, the power two in \eqref{eq:upper} is sharp whenever $\phi_3(0,0)\ne0$.  For $Z>1$ this applies, in particular, to the scalar two-electron ground state.
  \end{enumerate}
\end{theorem}

The four parts of Theorem~\ref{thm:main} constitute the main contributions of this paper.  Part~(i) shows that removing the universal pairwise cusp factor raises the global spectral Barron range from $s<1$ for $\psi$ to $s<2$ for both $\phi$ and $\phi_3$. Part~(ii) proves that the spectral Barron norm grows at most quadratically as $s\uparrow2$.

Parts~(iii) and~(iv) establish the sharpness of these two estimates. Part~(iii) shows that the regularity range in Part~(i) cannot be improved within the entire class of universal state-independent multiplicative factorizations specified there.  Part~(iv) proves for an unperturbed two-electron bound state that
\[
  \|\phi_3\|_{\B^{2-\varepsilon}(\R^6)}
  =32\pi Z|\phi_3(0,0)|\varepsilon^{-2}+\mathcal{O}(\varepsilon^{-1}).
\]
Thus the exponent two in Part~(ii) cannot be reduced whenever $\phi_3(0,0)\ne0$. This includes the positive scalar ground states of helium $(Z=2)$ and the helium-like ions $\mathrm{Li}^{+}$ $(Z=3)$ and $\mathrm{Be}^{2+}$ $(Z=4)$.

Only recently did we become aware of the independent work of Ehrlacher \cite[Theorem~4.1]{Ehrlacher2026}, who proved the same range $s<2$ for bound-state quotients associated with a prescribed cut-off Jastrow factor. For the hydrogen-like example $\psi(x)=x_1e^{-Z|x|/4}$, Ehrlacher \cite[Theorem~3.3 and Lemma~3.2]{Ehrlacher2026} showed that the associated quotient does not belong to $\B^2$ and that its $\B^{2-\varepsilon}$ norm grows like $\varepsilon^{-1}$ as $\varepsilon\downarrow0$. Thus Ehrlacher proves endpoint sharpness for one prescribed factor. Our results go further in both scope and precision. Part~(iii) rules out endpoint improvement for the entire stated class of universal state-independent factors. Part~(ii) establishes the quadratic upper bound for the general quotients, while Part~(iv) proves that this rate is sharp by deriving the exact $\varepsilon^{-2}$ asymptotic for physical two-electron eigenfunctions and showing that it is attained by the ground states of helium and helium-like ions.

The remainder of the paper is organized as follows. We establish the Fourier--Lebesgue bootstrap and verify the required Coulombic coefficient estimates in~\S~\ref{sec:bootstrap} and \S~\ref{sec:coefficients}, respectively. \S~\ref{sec:quantitative} proves the quantitative endpoint upper bound. We clarify the universal-factor obstruction in \S~\ref{sec:one-electron}, and \S~\ref{sec:sharpness} derives the exact two-electron asymptotic expansion.
\section{The Fourier--Lebesgue bootstrap}\label{sec:bootstrap}
\subsection{Fourier spaces}

For $t\in\R$ and $1\leq p\leq\infty$, let
\[ \FL_t^p(\R^d) =\left\{ u\in\mathcal S'(\R^d): \ip{\cdot}^{t}\widehat u\in L^p(\R^d) \right\}. \]
We equip this space with the norm
\[ \|u\|_{\FL_t^p(\R^d)} :=\|\ip{\cdot}^{t}\widehat u\|_{L^p(\R^d)}. \]
Thus $H^t(\R^d)=\FL_t^2(\R^d)$ and $\B^t(\R^d)=\FL_t^1(\R^d)$ isometrically.  We omit the ambient Euclidean space from a norm only when it is unambiguous.  The argument below only uses $1\leq p\leq2$.
Throughout the rest of the paper,
\[ \res=(-\Delta+1)^{-1}. \]

The purpose of the first lemma is to make the collision geometry explicit. The possible non-Wiener part of a Coulomb coefficient on $\R^d$ is not a genuinely $d$-dimensional multiplier. It has the form
\begin{equation}\label{eq:structured-block-expanded} a(\mathsf Mx-c), \qquad \mathsf M:\R^d\longrightarrow\R^m\text{ surjective}, \end{equation}
where $m=3$ for one collision coordinate and $m\leq6$ for two collision coordinates.  The remaining $d-m$ variables are spectators.

\begin{lemma}[Low-rank multiplier estimate]  \label{lem:multiplier} Let $1\leq p\leq2$, $\beta>0$, and suppose that $a\in L^\infty(\R^m)$ has a Fourier decomposition $\widehat a=g+h$ with $g\in L^1(\R^m)$ and $h\in L^{\alpha'}(\R^m)$ for some
  \begin{equation}\label{eq:alpha-conditions-expanded} \alpha>\max\{p,m/2\},\qquad 2\alpha\beta>m,\qquad \alpha'=\frac{\alpha}{\alpha-1}. \end{equation}
  Then, for every $u\in\FL_0^p(\R^d)$,
  \begin{equation}\label{eq:multiplier-norm-expanded}
    \|a(\mathsf M\cdot-c)u\|_{\FL_{-2\beta}^p}
    \leq C_{\mathsf M}(2\pi)^{-m/2}
    \left(
    \|g\|_{L^1(\R^m)}
    +\|\ip{\cdot}^{-2\beta}\|_{L^\alpha(\R^m)}
    \|h\|_{L^{\alpha'}(\R^m)}
    \right)
    \|u\|_{\FL_0^p}.
  \end{equation}
  Here $C_{\mathsf M}$ depends only on the fixed linear change of variables.  The bound is independent of the translation $c$.
\end{lemma}

\begin{proof}
First suppose that $\mathsf M(y,z)=y$, where $y\in\R^m$ and $z\in\R^{d-m}$, and put $v_\omega(\eta)=\widehat u(\eta,\omega)$ for each spectator frequency $\omega\in\R^{d-m}$.  The unitary product formula gives
  \[ \widehat{au}(\eta,\omega) =(2\pi)^{-m/2}[(g+h)*v_\omega](\eta). \]
  We apply the weighted H\"older--Young inequality from our earlier work \cite[Lemma~3.4]{MingYu2025} with $n=m$, $f_1=g$, $f_2=h$, and $\varphi=v_\omega$.  The condition $\alpha>m/2$ is part of \eqref{eq:alpha-conditions-expanded}; since $\alpha>p$, the parameter $\sigma=(1-\alpha/p)_+$ in the cited lemma is zero; and its condition $\beta>m/(2\alpha)$ is exactly $2\alpha\beta>m$.
  Because $\ip{(\eta,\omega)}^{-2\beta}\leq\ip\eta^{-2\beta}$, the cited lemma yields the fibrewise estimate in \eqref{eq:multiplier-norm-expanded}.
  Raising it to the $p$th power and integrating in $\omega$ proves the result when $\mathsf M$ is the coordinate projection.

  For a general surjective $\mathsf M$, choose an invertible linear map $T$ such that $\mathsf MT^{-1}(y,z)=y$.  Fourier transformation under $x=T^{-1}(y,z)$ introduces only the determinant of $T$ and the comparison constants between $\ip\xi$ and $\ip{T^T\xi}$; their product is denoted by $C_{\mathsf M}$.  Replacing $a(y)$ by $a(y-c)$ contributes only a unimodular Fourier phase and hence leaves the relevant norms unchanged.

  For completeness, the product identity is first applied to Schwartz functions.  If $u_j\to u$ in $\FL_0^p$, then Hausdorff--Young gives $u_j\to u$ in $L^{p'}$ because $1\leq p\leq2$.  Since $a(\mathsf M\cdot-c)\in L^\infty$, the products converge in $L^{p'}$ and hence in $\mathcal S'$.  The fibrewise estimate identifies this distributional limit with the limit in $\FL_{-2\beta}^p$, completing the density argument.
\end{proof}

\begin{lemma}[Elliptic lift and exponent descent]\label{lem:lift-descent}
  Let $0<\beta<1/2$ and $1\leq p\leq2$.
  \begin{enumerate}
    \item The resolvent satisfies the exact identity
          \begin{equation}\label{eq:resolvent-lift-expanded} \|(-\Delta+1)^{-1}f\|_{\FL_{2-2\beta}^p} =\|f\|_{\FL_{-2\beta}^p}. \end{equation}
    \item If $1<p\leq2$, $1\leq r<p$, and $1/r-1/p<(1-2\beta)/d$, then $\FL_{2-2\beta}^p(\R^d)$ embeds continuously into $\FL_1^r(\R^d)$.
  \end{enumerate}
\end{lemma}

\begin{proof}
  Part (i) follows from the exact identity $\ip\xi^2/(|\xi|^2+1)=1$.
 Part (ii) factors the missing weight $\ip\xi^{-(1-2\beta)}$ and applies H\"older's inequality.
  For the second assertion, define $\varkappa$ by $1/\varkappa=1/r-1/p.$
  The strict inequality $1/r-1/p<(1-2\beta)/d$ is equivalent to $\varkappa(1-2\beta)>d$. Therefore
  \begin{align*}
    \|u\|_{\FL_1^r}
     & =\|\ip\xi^{-(1-2\beta)}
    \ip\xi^{2-2\beta}\widehat u(\xi)\|_{L^r(\R^d_\xi)} \\
     & \leq
    \|\ip{\cdot}^{-(1-2\beta)}\|_{L^{\varkappa}(\R^d)}
    \|u\|_{\FL_{2-2\beta}^p},
  \end{align*}
  where we used H\"older's inequality. The proof is completed.
\end{proof}

\begin{lemma}[Structured Wiener multipliers]\label{lem:structured-wiener}
  Let $1\leq p\leq2$, let $\mathsf M:\R^d\to\R^m$ be surjective, and suppose that $\widehat w\in L^1(\R^m)$. Then, for every $c\in\R^m$,
  \begin{equation}\label{eq:structured-wiener}
    \|w(\mathsf M\cdot-c)f\|_{\FL_0^p}
    \leq C_{\mathsf M}(2\pi)^{-m/2}
    \|\widehat w\|_{L^1(\R^m)}\|f\|_{\FL_0^p}.
  \end{equation}
  Consequently, for any finite family of such factors,
  \[
    \left\|\prod_{\ell=1}^r
    w_\ell(\mathsf M_\ell\cdot-c_\ell)f\right\|_{\FL_0^p}
    \leq
    \prod_{\ell=1}^r
    \left[C_{\mathsf M_\ell}(2\pi)^{-m_\ell/2}
    \|\widehat w_\ell\|_{L^1(\R^{m_\ell})}\right]
    \|f\|_{\FL_0^p}.
  \]
  If a further factor $a_0(\mathsf M_0\cdot-c_0)$ satisfies the hypotheses of Lemma~\ref{lem:multiplier}, then the same product maps $\FL_0^p$ to $\FL_{-2\beta}^p$, with operator norm bounded by the right-hand constant in \eqref{eq:multiplier-norm-expanded} times the preceding product of Wiener constants.
\end{lemma}

\begin{proof}
  For a coordinate projection, \eqref{eq:structured-wiener} follows from Young's inequality on each active-frequency fibre and integration in the spectator frequency. The linear change of variables and density argument in the proof of Lemma~\ref{lem:multiplier} give the estimate for a general surjective map; translation contributes only a unimodular Fourier phase. Iterating \eqref{eq:structured-wiener} proves the product bound. Since the factors commute, all Wiener factors may be applied before the single possible non-Wiener factor, to which Lemma~\ref{lem:multiplier} then applies.
\end{proof}

\begin{proposition}[Finite Fourier-exponent bootstrap] \label{prop:bootstrap}
  Suppose $u\in H^1(\R^d)$ solves
  \begin{equation}\label{eq:abstract-equation} (-\Delta+1)u=\sum_{j=1}^dA_j\partial_j u+Cu \quad\text{in }\mathcal D'(\R^d), \end{equation}
  where every $A_j$ and $C$ is a finite linear combination of products
  \[
    a_0(\mathsf M_0x-c_0)
    \prod_{\ell=1}^r w_\ell(\mathsf M_\ell x-c_\ell),
    \qquad r\geq0.
  \]
  The leading factor is interpreted as $1$ when $a_0\equiv1$. Every map occurring in a product is surjective, and $\widehat w_\ell\in L^1(\R^{m_\ell})$. Otherwise, $a_0\in L^\infty(\R^{m_0})$ is the unique non-Wiener factor and, for every $0<\beta<1/2$ and $1\leq p\leq2$, its active Fourier transform satisfies the hypotheses of Lemma~\ref{lem:multiplier} for some $\alpha>\max\{p,m_0/2\}$ with $2\alpha\beta>m_0$. Then
  \[ u\in\B^s(\R^d)\qquad\text{for every }s<2. \]
\end{proposition}

\begin{proof}
  The bootstrap argument is carried out as follows.   Starting from $H^1=\FL_1^2$, the iteration first applies all structured Wiener factors without Fourier-weight loss and then applies the low-rank multiplier estimate once to the single possible non-Wiener factor. The resolvent gains almost two weighted orders and lowers the Fourier exponent while retaining one weighted order. The explicitly defined finite exponent chain reaches $p=1$; one final elliptic step gives every order below two.

  Fix $s<2$. Choose $\beta$ so that
  \begin{equation}\label{eq:beta-choice-expanded} 0<\beta<\frac12, \qquad s<2-2\beta. \end{equation}
  We now specify the finite exponent chain.  Put
  \[ \delta=\frac{1-2\beta}{2d}, \qquad \frac1{p_j}=\min\left\{\frac12+j\delta,1\right\}, \]
  and let $J$ be the first index for which $p_J=1$. Then $J$ is finite and, for every $j<J$,
  \[ 0<\frac1{p_{j+1}}-\frac1{p_j} \leq\delta<\frac{1-2\beta}{d}. \]

  The initial assumption is $u\in H^1=\FL_1^{p_0}$. Suppose inductively that $u\in\FL_1^{p_j}$. Then
  \[ \|\partial_k u\|_{\FL_0^{p_j}} \leq\|u\|_{\FL_1^{p_j}}, \qquad \|u\|_{\FL_0^{p_j}} \leq\|u\|_{\FL_1^{p_j}}. \]
  At each stage, the iteration estimates every product $A_k\partial_k u$ and $Cu$ by first applying Lemma~\ref{lem:structured-wiener} successively to all Wiener factors. If the coefficient contains a non-Wiener factor, Lemma~\ref{lem:multiplier} is then applied to that factor; otherwise we use $\FL_0^{p_j}\subset\FL_{-2\beta}^{p_j}$. Since all coefficient sums are finite, the right-hand side of \eqref{eq:abstract-equation} belongs to $\FL_{-2\beta}^{p_j}$. The resolvent lift gives $u\in\FL_{2-2\beta}^{p_j}$, and the exponent-descent part of Lemma~\ref{lem:lift-descent} gives $u\in\FL_1^{p_{j+1}}$. Iterating from $j=0$ to $j=J-1$ proves
  \[ u\in\FL_1^1=\B^1. \]
  One final application of the same coefficient estimates and the resolvent lift, now with $p=1$, yields $u\in\B^{2-2\beta}$. Because of \eqref{eq:beta-choice-expanded}, $\B^{2-2\beta}\subset\B^s$, which completes the proof.
\end{proof}

\begin{lemma}[Weighted convolution and translation estimates]
  Let $0<\varepsilon\leq1/2$ and let $G:\R^d\to\mathbb C^m$ be measurable. If $k$ satisfies $\ip{\cdot}^{\varepsilon}k\in L^1(\R^d)$, then
  \begin{equation}\label{eq:weighted-convolution}
    \int_{\R^d}\ip\Xi^{-\varepsilon}|k*G(\Xi)|\dd\Xi
    \leq2^{\varepsilon/2}\|\ip{\cdot}^{\varepsilon}k\|_{L^1(\R^d)}
    \int_{\R^d}\ip\Xi^{-\varepsilon}|G(\Xi)|\dd\Xi.
  \end{equation}
  If, in addition, $G\in W^{1,1}_{\mathrm{loc}}(\R^d)$ and $\ip{\cdot}^{-\varepsilon}\nabla G\in L^1$, then, for every $\Theta\in\R^d$,
  \begin{equation}\label{eq:weighted-translation}
    \int_{\R^d}\ip\Xi^{-\varepsilon}|G(\Xi-\Theta)-G(\Xi)|\dd\Xi
    \leq2^{\varepsilon/2}|\Theta|\ip\Theta^{\varepsilon}
    \int_{\R^d}\ip\Xi^{-\varepsilon}|\nabla G(\Xi)|\dd\Xi.
  \end{equation}
  Consequently, if $\int_{\R^d}k(\Theta)\dd\Theta=1$ and $|\cdot|\ip{\cdot}^{\varepsilon}k\in L^1$, then
  \begin{equation}\label{eq:weighted-localization}
    \int_{\R^d}\ip\Xi^{-\varepsilon}|k*G-G|(\Xi)\dd\Xi
    \leq2^{\varepsilon/2}\||\cdot|\ip{\cdot}^{\varepsilon}k\|_{L^1(\R^d)}
    \int_{\R^d}\ip\Xi^{-\varepsilon}|\nabla G(\Xi)|\dd\Xi.
  \end{equation}
\end{lemma}

\begin{proof}
  Peetre's inequality, written as
  \[ \ip{Y+\Theta}^{-\varepsilon}\leq2^{\varepsilon/2}\ip Y^{-\varepsilon}\ip\Theta^{\varepsilon}, \]
  and Tonelli's theorem give \eqref{eq:weighted-convolution} after the change of variables $Y=\Xi-\Theta$. The fundamental theorem of calculus gives
  \[ G(\Xi-\Theta)-G(\Xi)=-\int_0^1\Theta\cdot\nabla G(\Xi-t\Theta)\dd t. \]
  Applying the same Peetre inequality with $Y=\Xi-t\Theta$ proves \eqref{eq:weighted-translation}. Finally,
  \[ k*G-G=\int_{\R^d}k(\Theta)\{G(\,\cdot-\Theta)-G\}\dd\Theta \]
  when $\int_{\R^d}k(\Theta)\dd\Theta=1$; integrating \eqref{eq:weighted-translation} against $|k(\Theta)|\dd\Theta$ proves \eqref{eq:weighted-localization}.
\end{proof}

\section{Fourier structure of the Coulombic coefficients}\label{sec:coefficients}

Throughout this section
\[ J(x)=\chi(|x|)|x|, \qquad b(x)=\nabla J(x), \]
where $\chi$ is the fixed radial cut-off from Section~\ref{sec:introduction}. We use ``Wiener profile'' to mean a function whose Fourier transform belongs to $L^1$.
This section identifies the Fourier structure of the coefficients generated by the two conjugations. We first establish a localization principle showing that a cut-off preserves the explicit high-frequency Fourier term of a homogeneous or log-homogeneous singularity up to a rapidly decaying remainder, and then apply it to the angular field $b$, the Laplacian remainder $d_J=\Delta(J-|\cdot|)$, and the logarithmic profile $q$ and its derivatives. Products of two angular profiles are classified according to the rank of their collision coordinates, which distinguishes overlapping three-dimensional configurations from independent six-dimensional ones and identifies the corresponding critical Fourier scales. Every coefficient is thereby written as a finite sum of products in which all Wiener factors act without Fourier-weight loss and the single possible non-Wiener critical block has at most six active variables. The active Fourier transform of that block belongs to $L^1+L^r$ for every $r>1$. This verifies Proposition~\ref{prop:bootstrap} and proves Theorem~\ref{thm:main}(i) for both $\phi$ and $\phi_3$. The sharper high-frequency bounds also yield the five profile estimates used in Section~\ref{sec:quantitative} for Theorem~\ref{thm:main}(ii) and are reused in Section~\ref{sec:sharpness} in the proof of Theorem~\ref{thm:main}(iv).

\begin{lemma}[Localization of homogeneous and log-homogeneous profiles]
  \label{lem:homogeneous-localization} Let $h\in L^1_{\mathrm{loc}}(\R^m)$ be smooth on $\R^m\setminus\{0\}$ and suppose that, for some $\lambda\in\R$, an integer $k\geq0$, and functions $a_0,\ldots,a_k\in C^\infty(\mathbb S^{m-1})$,
  \begin{equation}\label{eq:log-homogeneous-scaling} h(t\omega)=t^\lambda\sum_{j=0}^k(\log t)^ja_j(\omega), \qquad t>0,\quad \omega\in\mathbb S^{m-1}. \end{equation}
  Let $\theta\in C_c^\infty(\R^m)$ equal one near the origin.  Assume that the Fourier transform of the tempered distribution $h$ is represented by a smooth function $H$ on $\R^m\setminus\{0\}$.  Then, for every integer $L>m+\lambda$, there is $C_L<\infty$ such that
  \begin{equation}\label{eq:localization-remainder} |\widehat{\theta h}(\xi)-H(\xi)|\leq C_L|\xi|^{-L}, \qquad |\xi|\geq1. \end{equation}
\end{lemma}

\begin{proof}

  Set $G=(\theta-1)h$. Since $\theta-1$ vanishes near the origin, $G$ is smooth on all of $\R^m$. The scaling law \eqref{eq:log-homogeneous-scaling} and the product rule imply, for every multi-index $\gamma$,
  \begin{equation*} |\partial^\gamma G(w)| \leq C_\gamma |w|^{\lambda-|\gamma|} (1+\log|w|)^k, \qquad |w|\geq R_\theta, \end{equation*}
  where $R_\theta$ contains the support of $\theta$.  If $|\gamma|=L>m+\lambda$, polar coordinates give
  \[ \int_{|w|\geq R_\theta}|\partial^\gamma G(w)|\dd w \leq C_\gamma|\mathbb S^{m-1}| \int_{R_\theta}^\infty r^{m-1+\lambda-L}(1+\log r)^k\dd r<\infty. \]
  The same derivative is integrable on the bounded region because $G$ is smooth.

  Fix $\xi\ne0$ and choose an index $\ell$ with $|\xi_\ell|\geq|\xi|/\sqrt m$. In the sense of tempered distributions,
  \[ (i\xi_\ell)^L\widehat G(\xi) =\widehat{\partial_\ell^LG}(\xi). \]
  The right-hand side is a bounded continuous function. Its modulus is at most
  \[ (2\pi)^{-m/2}\|\partial_\ell^LG\|_{L^1(\R^m)}. \]
  Therefore, away from $\xi=0$,
  \[ |\widehat G(\xi)| \leq m^{L/2}(2\pi)^{-m/2} \|\partial_\ell^LG\|_{L^1(\R^m)}|\xi|^{-L}. \]
  Finally, $\theta h=h+G$, so $\widehat{\theta h}=H+\widehat G$ away from the Fourier origin. This proves \eqref{eq:localization-remainder}. Any distribution supported at $\xi=0$ is annihilated when the equality is restricted to $\xi\ne0$ and hence has no bearing on the stated high-frequency estimate.
\end{proof}

\begin{lemma}[Angular field with an arbitrary cut-off]
  \label{lem:cutoff-angular-field} Let $\theta\in C_c^\infty(\R^3)$ equal one near the origin, and set $a_\theta(z)=\theta(z)z/|z|$ for $z\neq0$, with any value at $z=0$.
  Then, for $|k|\geq1$,
  \begin{equation} \label{eq:general-angular-tail} \left|\widehat a_\theta(k) +2i\sqrt{\frac2\pi}\frac{k}{|k|^4}\right| \leq C_\theta|k|^{-4}, \end{equation}
  and, on $\R^3$,
  \begin{equation} \label{eq:general-angular-global} |\widehat a_\theta(k)|\leq C_\theta\ip k^{-3}, \qquad |\nabla\widehat a_\theta(k)|\leq C_\theta\ip k^{-4}. \end{equation}
\end{lemma}

\begin{proof}
  For each component,
  \[ \widehat{z_j/|z|}(k) =-2i\sqrt{\frac2\pi}\frac{k_j}{|k|^4}, \qquad k\neq0. \]
  Lemma~\ref{lem:homogeneous-localization}, with $m=3$, degree zero, and $L=4$, proves \eqref{eq:general-angular-tail}.  Moreover,
  \[ \partial_{k_\ell}\widehat a_{\theta,j}(k) =-i\,\widehat{\theta(z)z_\ell z_j/|z|}(k). \]
  The uncut profile $z_\ell z_j/|z|$ is homogeneous of degree one.
  Its Fourier transform is therefore a smooth homogeneous function of degree $-4$ away from the origin.  Lemma~\ref{lem:homogeneous-localization}, now applied to this degree-one profile with $L=5$, gives
  \[ \left| \widehat{\theta(z)z_\ell z_j/|z|}(k) -\widehat{z_\ell z_j/|z|}(k) \right| \leq C_\theta|k|^{-5}, \qquad |k|\geq1. \]
  Consequently, $|\partial_{k_\ell}\widehat a_{\theta,j}(k)|\leq C_\theta|k|^{-4}$ for $|k|\geq1$.  Near the origin both transforms are bounded because $a_\theta,z_\ell a_{\theta,j}\in L^1$.  This proves \eqref{eq:general-angular-global}.
\end{proof}

\begin{lemma}[Angular, cut-off, and logarithmic profiles]\label{lem:profiles}
  The following statements hold.
  \begin{enumerate}
    \item There is a constant $C_b$, depending only on $\chi$, such that           \begin{equation}\label{eq:bdecay} \max_{1\leq j\leq3}|\widehat b_j(\xi)| \leq C_b\langle\xi\rangle^{-3}, \qquad \xi\in\R^3. \end{equation}
    \item If $d_J=\Delta(J-|\cdot|)$, then
          \begin{equation}\label{eq:dWiener} \widehat {d_J}\in L^1(\R^3). \end{equation}
    \item Let $q$ be defined by \eqref{eq:q}, and write $\Xi=(\xi,\eta)\in\R^6$. There are constants $C_q$ and a radial $\chi_\infty\in C^\infty(\R^6)$, equal to zero for $|\Xi|\leq1$ and to one for $|\Xi|\geq2$, such that
          \begin{align}
            |\widehat q(\Xi)|
             & \leq C_q\ip\Xi^{-8},                              \label{eq:q8} \\
            |\widehat{\partial_jq}(\Xi)|
             & \leq C_q\ip\Xi^{-7}, \qquad\text{for } 1\leq j\leq6,                             \label{eq:q7} \\
            \widehat{\Delta q}(\xi,\eta)
             & =-768\chi_\infty(\xi,\eta)
            \frac{\xi\cdot\eta}{(|\xi|^2+|\eta|^2)^4}
            +R_q(\xi,\eta),                                      \label{eq:q6}
          \end{align}
          where $R_q\in L^1(\R^6)\cap L^\infty(\R^6)$.  In particular, $\widehat{\nabla q}\in L^1(\R^6)$ and $\widehat{\Delta q}\in L^r(\R^6)$ for every $r>1$.
  \end{enumerate}
\end{lemma}

The high-frequency cut-off in \eqref{eq:q6} is part of the statement. Without $\chi_\infty$, the displayed homogeneous function is singular at the Fourier origin and cannot be combined with a global $L^1$ remainder.

\begin{proof}
  We prove the three assertions separately. For $b$ we isolate its degree-zero angular part. For $d_J$ we convert the noncompact Coulomb tail into the Fourier transform of a compactly supported Laplacian. For $q$ we compute the log-homogeneous leading transform and show that localization leaves an integrable remainder.

  \emph{Step 1. The angular field.}
  Differentiating $J$ gives the exact decomposition
  \begin{equation*} b_j(x)=\chi(|x|)\frac{x_j}{|x|}+\chi'(|x|)x_j. \end{equation*}
  The second summand is smooth and compactly supported, hence its Fourier transform is rapidly decreasing.  Lemma~\ref{lem:cutoff-angular-field}, applied with $\theta(x)=\chi(|x|)$, controls the first summand and proves \eqref{eq:bdecay}.

  \emph{Step 2. The Laplacian remainder.}
  The three-dimensional radial Laplacian yields
  \begin{equation}\label{eq:d-formula-expanded} d_J(x)=r\chi''(r)+4\chi'(r)+2\frac{\chi(r)-1}{r}, \qquad \text{for $r=|x|>0$}. \end{equation}
  The first two terms lie in $C_c^\infty(\R^3)$.  Put $g_J(x)=[\chi(|x|)-1]/|x|$.
  This function vanishes near the origin, equals $-|x|^{-1}$ for $|x|\geq2$, and tends to zero at infinity. A direct use of $\Delta f(r)=f''(r)+2f'(r)/r$ gives
  \[ \Delta g_J(x)=\chi''(|x|) |x|^{-1}=:w_J(x), \qquad w_J\in C_c^\infty(\R^3). \]
  There is no distribution supported at the Fourier origin, because $g_J(x)$ tends to zero rather than containing a nonzero polynomial. Hence $\widehat {g_J}(\xi)=-|\xi|^{-2}\widehat {w_J}(\xi)$.
  The factor $|\xi|^{-2}$ is integrable near the origin in dimension three, while $\widehat {w_J}$ decreases faster than every power at infinity.  Thus $\widehat {g_J}\in L^1$.  Formula \eqref{eq:d-formula-expanded} proves \eqref{eq:dWiener}.

  \emph{Step 3. The logarithmic profile.}
  Near $(0,0)$, the cut-offs in $q$ equal one and
  \[ q_0(x,y)=(x\cdot y)\log(|x|^2+|y|^2). \]
  Let $\theta\in C_c^\infty(\R^6)$ equal one on the support of the product $\chi(|x|)\chi(|y|)$. Then $q=\theta q_0+h_q$, where
  $h_q=(\chi(|x|)\chi(|y|)-\theta)q_0$ is smooth and compactly supported.   Hence $\widehat h_q$ is rapidly decreasing.

  Writing $\omega=(\omega_x,\omega_y)\in\mathbb S^5$, the scaling formula is
  \[ q_0(t\omega) =t^2(\omega_x\cdot\omega_y)\log(t^2|\omega|^2) =2t^2(\log t)(\omega_x\cdot\omega_y). \]
  Thus \eqref{eq:log-homogeneous-scaling} holds with $m=6$, $\lambda=2$, $k=1$, $a_0=0$, and $a_1(\omega)=2\omega_x\cdot\omega_y$.  The six-dimensional Riesz-distribution formula gives
  \[ \widehat{\log|(x,y)|^2}(\xi,\eta) =-16(|\xi|^2+|\eta|^2)^{-3}, \qquad (\xi,\eta)\neq(0,0). \]
  Multiplication by $x_jy_j$ becomes $-\partial_{\xi_j}\partial_{\eta_j}$; summing over $j=1,2,3$ therefore gives
  \begin{equation*} \widehat q_0(\xi,\eta) =768\frac{\xi\cdot\eta} {(|\xi|^2+|\eta|^2)^5}, \qquad (\xi,\eta)\ne(0,0), \end{equation*}
  up to distributions supported at the Fourier origin.  Applying Lemma~\ref{lem:homogeneous-localization} with $m=6$, $\lambda=2$, $k=1$, and $L=9$ yields
  \begin{equation}\label{eq:q-localization-error} \left| \widehat q(\xi,\eta) -768\frac{\xi\cdot\eta} {(|\xi|^2+|\eta|^2)^5} \right| \leq C|(\xi,\eta)|^{-9}, \qquad |(\xi,\eta)|\geq1. \end{equation}
Since $|\xi\cdot\eta|\leq|\Xi|^2$, \eqref{eq:q-localization-error} gives $|\widehat q(\Xi)|\leq C|\Xi|^{-8}$ for $|\Xi|\geq1.$
Moreover, $q\in L^1(\R^6)$, so $\widehat q$ is bounded on $|\Xi|\leq1$. This proves \eqref{eq:q8}. Since $\widehat{\partial_jq}(\Xi)   = i\Xi_j\widehat q(\Xi)$ and $\partial_jq\in L^1(\R^6)$, the same argument, followed by taking the maximum over $1\leq j\leq6$, proves \eqref{eq:q7}. Finally, $\ip{\cdot}^{-7}\in L^1(\R^6)$, and hence $\widehat{\nabla q}\in L^1(\R^6)$.

  Multiplication of \eqref{eq:q-localization-error} by $-|\Xi|^2$ gives, for $|\Xi|\geq2$,
  \begin{equation}\label{eq:lap-q-high-frequency-error} \left| \widehat{\Delta q}(\xi,\eta) +768\frac{\xi\cdot\eta}{(|\xi|^2+|\eta|^2)^4} \right|\leq C|\Xi|^{-7}. \end{equation}
  Define $R_q$ by \eqref{eq:q6}. Since $\chi_\infty=1$ on $|\Xi|\geq2$, \eqref{eq:lap-q-high-frequency-error} gives $|R_q(\Xi)|\leq C|\Xi|^{-7}$ for $|\Xi|\geq2.$ On $|\Xi|\leq2$, both terms defining $R_q$ are bounded. Indeed, the cut-off homogeneous term vanishes on $|\Xi|\leq1$ and is bounded on $1\leq|\Xi|\leq2$. Moreover, since $q\in L^1(\R^6)$, $|\widehat{\Delta q}(\Xi)|  =|\Xi|^2|\widehat q(\Xi)|$ is bounded. Hence $R_q\in L^\infty(\R^6)$ and
  \[ \|R_q\|_{L^1(\R^6)} \leq |B_6(0,2)|\|R_q\|_{L^\infty(B_6(0,2))} +C|\mathbb S^5|\int_2^\infty r^{-2}\dd r<\infty. \]
  This proves $R_q\in L^1\cap L^\infty$ and completes the proof.
\end{proof}

Consequently, the gradient of $F_{3,\mathrm{cut}}$ acts as a Wiener multiplier, while each term in its Laplacian contains a single critical radial scale in six collision variables.

\begin{lemma}[Products of two angular collision profiles]
  \label{lem:two-angular-blocks} For $\ell=1,2$, let $\gamma_\ell(x)=L_\ell x-c_\ell$ be a collision variable of the form $x_i-R_\nu$ or $x_i-x_j$, and set $m=\operatorname{rank}\bigl(x\mapsto(L_1x,L_2x)\bigr)\in\{3,6\}.$
  For any components $b_j,b_k$ of $b=\nabla J$, there exist a surjective linear map $\mathsf M:\R^{3N}\to\R^m$ and a profile $a\in L^\infty_c(\R^m)$ such that
  \[ b_j(\gamma_1(x))b_k(\gamma_2(x))=a(\mathsf Mx). \]
  The active profile can be chosen so that, in the rank-six case,
  \[ |\widehat a(\xi,\eta)| \leq C\ip\xi^{-3}\ip\eta^{-3}, \]
  whereas in the rank-three case
  \[ |\widehat a(\xi)|\leq C\ip\xi^{-3}. \]
  In particular, $\widehat a\in L^r(\R^m)$ for every $r>1$.
\end{lemma}

\begin{proof}
For each $\ell=1,2$, there is a nonzero row vector $\lambda_\ell\in\R^N$ such that $L_\ell x=\sum_{i=1}^N(\lambda_\ell)_i x_i.$ For an electron--nucleus variable, $\lambda_\ell=e_i$, whereas for an electron--electron variable, $\lambda_\ell=e_i-e_j$. Equivalently, $L_\ell=\lambda_\ell\otimes I_3$. Hence
\[
  \operatorname{rank}(L_1,L_2)
  =3\operatorname{rank}
  \begin{pmatrix}
    \lambda_1\\
    \lambda_2
  \end{pmatrix}
  \in\{3,6\}.
\]

  Suppose first that $m=6$. Then $\mathsf Mx=(L_1x,L_2x)$ is surjective, and we may take
  \[ a(y,z)=b_j(y-c_1)b_k(z-c_2), \qquad (y,z)\in\R^3\times\R^3. \]
  Since translations produce only unimodular Fourier factors, $|\widehat a(\xi,\eta)| =|\widehat b_j(\xi)|\,|\widehat b_k(\eta)|.$
  By \eqref{eq:bdecay}, we have $|\widehat b_\ell(\xi)| \leq C_b\langle\xi\rangle^{-3}$ for $1\leq\ell\leq3.$
  Consequently, $\widehat b_\ell\in L^r(\R^3)$ for every $r>1$.
  Fubini's theorem gives
  \[ \|\widehat a\|_{L^r(\R^6)} =\|\widehat b_j\|_{L^r(\R^3)} \|\widehat b_k\|_{L^r(\R^3)} <\infty. \]

Suppose next that $m=3$. Then $L_2=\sigma L_1$ for some $\sigma\in\{1,-1\}$. Set $y=L_1x-c_1$ and $y_0=\sigma c_2-c_1$.
Since $b$ is odd,
\[ b_j(L_1x-c_1)b_k(L_2x-c_2) =\sigma b_j(y)b_k(y-y_0). \]
The constant factor $\sigma$ does not affect the Fourier estimates.
It remains to distinguish the cases $y_0=0$ and $y_0\neq0$.

Assume first that $y_0\neq0$. Choose disjoint neighborhoods of $0$ and $y_0$ and a smooth partition of unity on the support of $a$. Near $0$, the factor $b_k(y-y_0)$ is smooth, so the corresponding localized term has the form
\[ A(y)b_j(y),\qquad A\in C_c^\infty(\R^3). \]
After translating $y_0$ to the origin, the term localized near $y_0$ has the same form with $b_j$ replaced by $b_k$. The remaining term is smooth and compactly supported.

We claim the local estimate that for every $A\in C_c^\infty(\R^3)$ and every component $b_q$ of $b$,
\begin{equation}\label{eq:smooth-times-angular-decay} |\widehat{A b_q}(\xi)|\leq C_A\langle\xi\rangle^{-3}. \end{equation}
To prove it, choose $\theta\in C_c^\infty(\R^3)$ equal to one near the origin and supported where $\chi=1$. Near the origin, we write $A(y)=A(0)+\nabla A(0)\cdot y+R_2(y)$ with $|\partial^\alpha R_2(y)|\leq C_\alpha |y|^{\max\{2-|\alpha|,0\}}$ for $|\alpha|\leq4.$
Since $b_q(y)=y_q/|y|$ on the support of $\theta$,
\[ A(y)b_q(y)=A(0)\theta(y)\frac{y_q}{|y|}+\theta(y)\bigl(\nabla A(0)\cdot y\bigr)\frac{y_q}{|y|}+\theta(y)R_2(y)\frac{y_q}{|y|}+S(y), \]
where $S\in C_c^\infty(\R^3)$. 
Lemma~\ref{lem:homogeneous-localization} gives Fourier decay $C|\xi|^{-3}$ for the first term and $C|\xi|^{-4}$ for the second. The derivative bounds on $R_2$ imply
\[ \theta(y)R_2(y)\frac{y_q}{|y|}\in W^{4,1}(\R^3). \]
Its Fourier transform is therefore bounded by $C\langle\xi\rangle^{-4}$. Since all four terms belong to $L^1$, their Fourier transforms are bounded near the origin. This proves \eqref{eq:smooth-times-angular-decay}.

Applying \eqref{eq:smooth-times-angular-decay} to the two singular pieces of the partition, and using rapid Fourier decay for the smooth remainder, gives $|\widehat a(\xi)|\leq C\langle\xi\rangle^{-3}.$

It remains to consider $y_0=0$. In this case $a(y)=\sigma b_j(y)b_k(y).$
Since $b_q(y)=y_q/|y|$ near the origin, there are $\theta\in C_c^\infty(\R^3)$, equal to one near the origin, and $S\in C_c^\infty(\R^3)$ such that
\[ a(y)=\sigma\theta(y)\frac{y_jy_k}{|y|^2}+S(y). \]
Decompose
\begin{equation*} \frac{y_jy_k}{|y|^2}=\frac{\delta_{jk}}{3}+\frac{H_{jk}(y)}{|y|^2},\qquad H_{jk}(y)=y_jy_k-\frac{\delta_{jk}}{3}|y|^2. \end{equation*}
The constant term is smooth. 
Since $\Delta H_{jk} =0,$ $H_{jk}$ is a homogeneous harmonic polynomial of degree two. Hence for $\omega\in\mathbb S^2,$ $Y_{jk}(\omega)=H_{jk}(\omega)$ is a spherical harmonic of degree two and $H_{jk}(y)/|y|^2 =Y_{jk}(y/{|y|})$ is homogeneous of degree zero. The Fourier transform formula for homogeneous spherical harmonics gives
\[ \widehat{\frac{H_{jk}}{|\cdot|^2}}(\xi) =-\frac{3\sqrt{2\pi}}2|\xi|^{-3} H_{jk}\left(\frac{\xi}{|\xi|}\right), \qquad \xi\neq0. \]
Lemma~\ref{lem:homogeneous-localization}, applied with $m=3$, $\lambda=0$, $k=0$, and $L=4$, therefore yields $|\widehat a(\xi)|\leq C|\xi|^{-3}$ for $|\xi|\geq1.$
Since $a\in L^1(\R^3)$, its Fourier transform is bounded for $|\xi|\leq1$. Thus $|\widehat a(\xi)|\leq C\langle\xi\rangle^{-3}.$

Consequently, in either rank-three case, $\widehat a\in L^r(\R^3)$ for every $r>1,$ since $\langle\xi\rangle^{-3}\in L^r(\R^3)$ exactly when $r>1$. This completes the proof.
\end{proof}

To avoid repeatedly renaming the same two factors in the conjugated equations below, we write
\[ F=F_{2,\mathrm{cut}},\qquad P=F_{3,\mathrm{cut}}. \]

\begin{proof}[Proof of Theorem~\ref{thm:main}(i)]
  We first derive the two conjugated equations exactly, including all signs and zeroth-order terms. We then classify every coefficient by the preceding profile lemmas. Once the low-rank multiplier hypothesis has been verified, the same Fourier--Lebesgue bootstrap applies to both quotients.

  We first verify that the two conjugation identities used below hold in $H^{-1}(\R^{3N})$, and hence in $\mathcal D'(\R^{3N})$, despite the nonsmooth pairwise cusp factor. For $\delta>0$, set $J_\delta(z)=\chi(|z|)(|z|^2+\delta^2)^{1/2}$, obtain $F_\delta$ from $F$ by replacing $J$ with $J_\delta$, and put $\phi_\delta=e^{-F_\delta}\psi$. For every $v\in H^1(\R^{3N})$, the codimension-three Hardy inequalities give
  \[
    \left\|\frac{v}{|x_i-R_\nu|}\right\|_2\leq2\|\nabla_{x_i}v\|_2,
    \qquad
    \left\|\frac{v}{|x_i-x_j|}\right\|_2^2
    \leq2\bigl(\|\nabla_{x_i}v\|_2^2+\|\nabla_{x_j}v\|_2^2\bigr).
  \]
  Moreover, uniformly in $\delta$,
  \[
    |\nabla F_\delta|\leq C,
    \qquad
    |\Delta F_\delta|+|V|
    \leq C\left(1+\sum_{i,\nu}\frac1{|x_i-R_\nu|}
    +\sum_{i<j}\frac1{|x_i-x_j|}\right).
  \]
  Since $F_\delta\to F$ uniformly and $\nabla F_\delta\to\nabla F$ almost everywhere, dominated convergence and the Hardy inequalities yield
  \[
    \phi_\delta\to\phi\quad\text{in }H^1,
    \qquad
    (\Delta F_\delta)\phi_\delta\to(\Delta F)\phi,
    \quad V\phi_\delta\to V\phi\quad\text{in }L^2.
  \]
  The classical product rule applied to $F_\delta$ gives
  \[
    (-\Delta+1)\phi_\delta
    =2\nabla F_\delta\cdot\nabla\phi_\delta
    +\bigl(\Delta F_\delta-V+|\nabla F_\delta|^2+E+1\bigr)\phi_\delta.
  \]
  Passing to the limit in $H^{-1}$ and using $\Delta F_2=V$ in $\mathcal D'$ gives
  \begin{equation}\label{eq:phi-equation} (-\Delta+1)\phi =2\nabla F\cdot\nabla\phi +\bigl(\Delta(F-F_2)+|\nabla F|^2+E+1\bigr)\phi. \end{equation}
  Conjugation by $e^{F_{2,\mathrm{cut}}}$ therefore replaces the uncut Coulomb potential by bounded angular coefficients and terms supported in the cut-off transition region.

We now record the coefficients in \eqref{eq:phi-equation} explicitly.
Recall that
\[ F=F_{2,\mathrm{cut}} =-\frac12\sum_{i=1}^N\sum_{\nu=1}^L Z_\nu J(x_i-R_\nu) +\frac14\sum_{1\leq i<j\leq N}J(x_i-x_j), \qquad b=\nabla J. \]
Since $b$ is odd, differentiation with respect to $x_i$ gives
\begin{equation}\label{eq:grad-F-explicit}
  \nabla_{x_i}F
  =-\frac12\sum_{\nu=1}^L Z_\nu b(x_i-R_\nu)
  +\frac14\sum_{\substack{1\leq j\leq N\\j\neq i}}
    b(x_i-x_j),
  \qquad 1\leq i\leq N.
\end{equation}
Next, let $d_J=\Delta(J-|\cdot|).$
For a function $f(x_i-x_j)$,
\[ (\Delta_{x_i}+\Delta_{x_j})f(x_i-x_j) =2(\Delta f)(x_i-x_j). \]
It follows that, in distributions,
\begin{equation}\label{eq:lap-F-remainder-explicit} \Delta(F-F_2) =-\frac12\sum_{i=1}^N\sum_{\nu=1}^L Z_\nu d_J(x_i-R_\nu) +\frac12\sum_{1\leq i<j\leq N}d_J(x_i-x_j). \end{equation}
Finally, \eqref{eq:grad-F-explicit} gives
\begin{equation} \label{eq:grad-F-square-explicit}
\begin{aligned}
  |\nabla F|^2
  &=\frac14\sum_{i=1}^N\sum_{\nu,\mu=1}^L
    Z_\nu Z_\mu\,
    b(x_i-R_\nu)\cdot b(x_i-R_\mu)   \\
  &\quad-\frac14\sum_{i=1}^N\sum_{\nu=1}^L
    \sum_{\substack{1\leq j\leq N\\j\neq i}}
    Z_\nu\,b(x_i-R_\nu)\cdot b(x_i-x_j)  \\
  &\quad+\frac1{16}\sum_{i=1}^N
    \sum_{\substack{1\leq j\leq N\\j\neq i}}
    \sum_{\substack{1\leq k\leq N\\k\neq i}}
    b(x_i-x_j)\cdot b(x_i-x_k).
\end{aligned}
\end{equation}

Equations \eqref{eq:grad-F-explicit}, \eqref{eq:lap-F-remainder-explicit}, and \eqref{eq:grad-F-square-explicit} express the coefficients in \eqref{eq:phi-equation} as finite sums of single angular profiles, products of two angular profiles, and Wiener profiles. Fix $0<\beta<1/2$ and $1\leq p\leq2$. By Lemma~\ref{lem:profiles}(i) and Lemma~\ref{lem:two-angular-blocks}, the Fourier transform of every corresponding $m$-dimensional non-Wiener profile, $m\in\{3,6\}$, belongs to $L^r(\R^m)$ for every $r>1$. Choosing $\alpha>\max\left\{p,m/(2\beta)\right\}$, and taking $g=0$ and $h$ equal to that Fourier transform verifies Lemma~\ref{lem:multiplier}. For the copies of $d_J$, Lemma~\ref{lem:profiles}(ii) gives $\widehat {d_J}\in L^1(\R^3)$, so Lemma~\ref{lem:structured-wiener} applies. The constant term $E+1$ is a bounded multiplier on every $\FL_t^p$.
Thus all hypotheses of Proposition~\ref{prop:bootstrap} hold for \eqref{eq:phi-equation}. Since $\phi\in H^1(\R^{3N})$, the proposition yields $\phi\in\B^s(\R^{3N})$ for every $s<2.$ 

  We next verify the second conjugation in the same weak sense. Equations \eqref{eq:q8} and \eqref{eq:q7} imply $\widehat q,\widehat{\partial_jq}\in L^1(\R^6)$ for $1\leq j\leq6$, hence $q,\nabla q\in L^\infty(\R^6)$. For $h(x,y)=(x\cdot y)\log(|x|^2+|y|^2)$, direct differentiation gives
  \[
    |D^2h(X)|\leq C(1+|\log|X||),
    \qquad
    (\Delta_x+\Delta_y)h=16\frac{x\cdot y}{|x|^2+|y|^2}
  \]
  away from the origin. Since $h\in W^{2,1}_{\mathrm{loc}}(\R^6)$, the second identity also holds in distributions and has a bounded right-hand side. The cut-off derivatives are supported away from the origin; hence
  \[
    P\in W^{1,\infty}(\R^{3N}),
    \qquad \Delta P\in L^\infty(\R^{3N}).
  \]
  It follows that $e^{-P}\in W^{1,\infty}$ and $\phi_3=e^{-P}\phi\in H^1$. Mollifying $P$ and passing to the limit in $H^{-1}$ gives, for every $v\in H^1$,
  \[
    \Delta(e^Pv)=e^P\bigl(\Delta v+2\nabla P\cdot\nabla v
    +(\Delta P+|\nabla P|^2)v\bigr)
    \quad\text{in }H^{-1}.
  \]
  Applying this identity to $v=\phi_3$ in \eqref{eq:phi-equation} and collecting the two drift terms gives
\begin{equation} \label{eq:phi3-equation}
\begin{aligned}
(-\Delta+1)\phi_3
     & =2\nabla(F+P)\cdot\nabla\phi_3 \\
     & \quad+\bigl(
    \Delta(F-F_2)+\Delta P+|\nabla(F+P)|^2+E+1
    \bigr)\phi_3.
\end{aligned}
\end{equation}
  By \eqref{eq:q7}, each component of $\nabla P$ is a finite sum of structured Wiener factors, whereas \eqref{eq:q6} shows that each block of $\Delta P$ is a six-dimensional non-Wiener block satisfying Lemma~\ref{lem:multiplier}. After expanding $|\nabla(F+P)|^2$, every coefficient in \eqref{eq:phi3-equation} is a finite sum of products of the form required by Proposition~\ref{prop:bootstrap}: all Wiener factors are controlled successively by Lemma~\ref{lem:structured-wiener}, and the single possible non-Wiener factor is controlled by Lemma~\ref{lem:multiplier}. Since $\phi_3\in H^1(\R^{3N})$, Proposition~\ref{prop:bootstrap} yields $\phi_3\in\B^s(\R^{3N})$ for every $s<2$.
\end{proof}

\section{The quantitative endpoint upper bound}\label{sec:quantitative}

This section proves Theorem~\ref{thm:main}(ii) by tracking the dependence on $\varepsilon=2-s$ with the Fourier exponent fixed at one. We first choose a H\"older exponent that is admissible for every possible non-Wiener block, whose active dimension is at most six, and show that the resulting negative-weight norm is uniformly bounded as $\varepsilon\downarrow0$. We then turn the Fourier classification from Section~\ref{sec:coefficients} into quantitative multiplier bounds. A single critical radial scale contributes $\varepsilon^{-1}$, two independent angular scales contribute at most $\varepsilon^{-2}$, and Wiener profiles remain bounded. These estimates are inserted term by term into the two conjugated equations and grouped according to their pole order. The resolvent gives the required $\B^{2-\varepsilon}$ estimate, while a separate low--high frequency argument with a fixed Fourier-weight loss of $1/2$ controls the $\B^1$ norm independently of $\varepsilon$. Besides proving the quantitative upper bound, this analysis provides the pole classification used in Section~\ref{sec:sharpness}.

For $0<\varepsilon\leq1/2$, set
\begin{equation*}
  \alpha_\varepsilon=\frac{6(1+\sqrt\varepsilon)}{\varepsilon},
  \qquad
  r_\varepsilon=\alpha_\varepsilon'
  =1+\delta_\varepsilon,
  \qquad
  \delta_\varepsilon=\frac{\varepsilon}
  {6(1+\sqrt\varepsilon)-\varepsilon}.
\end{equation*}
Thus $\alpha_\varepsilon\varepsilon=6(1+\sqrt\varepsilon)>6$, which exceeds the active dimension of every possible non-Wiener critical block.

For a block $a(\mathsf Mx-c)$ as in \eqref{eq:structured-block-expanded}, define
\begin{equation}\label{eq:m-def-expanded}
  \mathfrak m_{\varepsilon,\mathsf M}(a)
  =C_{\mathsf M}(2\pi)^{-m/2}
  \inf_{\widehat a=g+h}
  \left(
  \|g\|_{L^1(\R^m)}
  +W_{m,\varepsilon}\|h\|_{L^{r_\varepsilon}(\R^m)}
  \right),
\end{equation}
where the infimum is taken over $g\in L^1(\R^m)$ and $h\in L^{r_\varepsilon}(\R^m)$, and
\begin{equation*} W_{m,\varepsilon} =\|\ip{\cdot}^{-\varepsilon}\|_{L^{\alpha_\varepsilon}(\R^m)}. \end{equation*}
For each scalar profile $a$ used below, let $\mathcal M(a)$ be the finite family of linear parts of the affine collision maps through which $a$ occurs in the two conjugated equations, and set
\begin{equation}\label{eq:m-envelope}
  \mathfrak m_\varepsilon(a)
  :=\max_{\mathsf M\in\mathcal M(a)}
  \mathfrak m_{\varepsilon,\mathsf M}(a).
\end{equation}
The estimate is independent of the translation $c$, so \eqref{eq:m-envelope} covers all translates of a given profile; a maximum over an empty family is understood as zero. We apply the definition separately to the components $b_j$ and $\partial_jq$. Lemma~\ref{lem:multiplier}, with $p=1$ and $2\beta=\varepsilon$, says precisely that
\begin{equation}\label{eq:m-operator-expanded} \|a(\mathsf M\cdot-c)f\|_{\FL_{-\varepsilon}^1} \leq\mathfrak m_{\varepsilon,\mathsf M}(a)\|f\|_{\FL_0^1}. \end{equation}

\subsection{The negative Fourier weight is uniform}

The purpose of this subsection is to prove that the auxiliary factor $W_{m,\varepsilon}$ in \eqref{eq:m-def-expanded} is uniformly bounded for every non-Wiener active dimension $1\leq m\leq6$.
Hence the negative weight creates no endpoint pole, and all singular dependence on $\varepsilon$ in the next subsection comes from the Fourier tails of the coefficients.

By direct calculations or by \cite[Lemma 3.4]{MingYu2025},
\begin{equation}\label{eq:weight-constant} W_{m,\varepsilon} =\left[ \pi^{m/2} \frac{\Gamma((6(1+\sqrt\varepsilon)-m)/2)} {\Gamma(3(1+\sqrt\varepsilon))} \right]^{\varepsilon/[6(1+\sqrt\varepsilon)]}. \end{equation}
For $m=6$, the Gamma recurrence $\Gamma(z+3)=z(z+1)(z+2)\Gamma(z)$ gives the exact expression
\begin{equation}\label{eq:W6-exact-expanded} W_{6,\varepsilon} =\left[ \frac{\pi^3}{3\sqrt\varepsilon(1+3\sqrt\varepsilon) (2+3\sqrt\varepsilon)} \right]^{\varepsilon/[6(1+\sqrt\varepsilon)]}. \end{equation}
This tends to one because its logarithm is
\[ \frac{\varepsilon}{6(1+\sqrt\varepsilon)} \log\frac{\pi^3}{3\sqrt\varepsilon(1+3\sqrt\varepsilon) (2+3\sqrt\varepsilon)}, \]
and $\varepsilon|\log\varepsilon|\to0$. For $m<6$, the Gamma quotient in \eqref{eq:weight-constant} is continuous and finite at $\varepsilon=0$. Therefore the finite number
\begin{equation*} C_W=\sup_{0<\varepsilon\leq1/2} \max_{1\leq m\leq6}W_{m,\varepsilon} \end{equation*}
is well defined.  Thus the weight $W_{m,\varepsilon}$ in \eqref{eq:m-operator-expanded} produces no endpoint pole.

\subsection{The five profile estimates}

This subsection converts the qualitative Fourier information from Section~\ref{sec:coefficients} into the five uniform multiplier bounds needed below. The high-frequency integral for one angular block is evaluated explicitly and produces a simple pole. Products of two independent angular blocks factor into two such integrals and may therefore produce a double pole, whereas overlapping blocks retain only one critical radial scale. The six-dimensional radial tail of $\Delta q$ also gives a simple pole. By contrast, $d_J$ and the components of $\nabla q$ are Wiener profiles and remain uniformly bounded, while a product of an angular and a Wiener profile inherits only the angular simple pole. Taking maxima over the finitely many collision maps determined by the fixed system yields the five constants used in the term-by-term equation estimate.

For the angular profile define the finite cut-off constants
\begin{align*}
  L_b     & =\max_{1\le j\le 3}
  \|\mathbf1_{|\xi|\leq1}\widehat b_j\|_{L^1(\R^3)}, \\
  K_b     & =\max_{1\le j\le 3}\operatorname*{ess\,sup}_{|\xi|>1}
  |\xi|^3|\widehat b_j(\xi)|.
\end{align*}
Both are finite by Lemma~\ref{lem:profiles}.  Since $r_\varepsilon=1+\delta_\varepsilon$, polar coordinates give the exact high-frequency integral
\begin{equation} \label{eq:b-r-exact-expanded}
\begin{aligned}
\int_{|\xi|>1}|\widehat b_j(\xi)|^{r_\varepsilon}\dd\xi
& \leq4\pi K_b^{r_\varepsilon}   \int_1^\infty t^{2-3r_\varepsilon}\dd t \\
& =\frac{4\pi K_b^{r_\varepsilon}}{3\delta_\varepsilon}.       
\end{aligned}
\end{equation}
Splitting $\widehat b_j$ at $|\xi|=1$ in \eqref{eq:m-def-expanded} therefore gives
\begin{equation}\label{eq:b-m-explicit-expanded} \mathfrak m_\varepsilon(b_j) \leq \max_{\mathsf M\in\mathcal M(b_j)}C_{\mathsf M}(2\pi)^{-3/2} \left[ L_b+C_WK_b \left(\frac{4\pi}{3\delta_\varepsilon}\right)^{1/r_\varepsilon} \right]. \end{equation}
The product $\varepsilon\mathfrak m_\varepsilon(b_j)$ is bounded on $(0,1/2]$ because
\begin{equation}\label{eq:eps-delta-exact-expanded} \frac{\varepsilon}{\delta_\varepsilon} =6(1+\sqrt\varepsilon)-\varepsilon, \end{equation}
 Formula \eqref{eq:b-r-exact-expanded} identifies the only divergent integral. One angular block has one critical radial scale.

For two independent angular blocks, Lemma~\ref{lem:two-angular-blocks} reduces the active profile to a tensor product on $\R^3\times\R^3$; hence
\begin{equation*} \|\widehat b_j\otimes\widehat b_k\|_{L^{r_\varepsilon}(\R^6)} =\|\widehat b_j\|_{L^{r_\varepsilon}(\R^3)} \|\widehat b_k\|_{L^{r_\varepsilon}(\R^3)}. \end{equation*}
The two copies of \eqref{eq:b-r-exact-expanded} give at most $\varepsilon^{-2}$.  If the active dimension is three, the second conclusion of Lemma~\ref{lem:two-angular-blocks} gives $|\widehat a(\xi)|\leq C\ip\xi^{-3}$ directly.  The single radial integral in \eqref{eq:b-r-exact-expanded} then gives only a simple pole.

For $\Delta q$, put $\varrho=(|\xi|^2+|\eta|^2)^{1/2}$ and
\[ A_r=\int_{\mathbb S^5}|\omega_\xi\cdot\omega_\eta|^r\dd\omega. \]
The sphere decomposition $\omega=(\sqrt t,a,\sqrt{1-t},c)$ gives, for $r>-1$,
\begin{equation*} A_r=\frac{8\pi^2}{r+1} \frac{\Gamma((r+3)/2)^2}{\Gamma(r+3)}. \end{equation*}
In particular $A_1=2\pi^2/3$.  Since $\chi_\infty=1$ on $\{\varrho\geq2\}$, for every $R\geq2$ the cut-off homogeneous term in \eqref{eq:q6} satisfies the exact radial identity
\begin{equation}\label{eq:dq-radial-expanded} \int_{\{(\xi,\eta)\in\R^6:\,\varrho>R\}} \left|768\frac{\xi\cdot\eta}{\varrho^8}\right|^r \dd\xi\dd\eta =768^rA_r\frac{R^{-6(r-1)}}{6(r-1)}. \end{equation}
There is only one six-dimensional radial scale, so this is a simple pole. Denote the cut-off homogeneous term in \eqref{eq:q6} by
\[
  H_q(\xi,\eta)=-768\chi_\infty(\xi,\eta)
  \frac{\xi\cdot\eta}{\varrho^8}.
\]
The transition annulus $1<\varrho<2$ contributes a uniformly bounded term. More precisely, for $1\leq r\leq r_{1/2}$,
\[
  \|H_q\|_{L^r(\R^6)}^r
  =768^rA_r\frac{2^{-6(r-1)}}{6(r-1)}+E_r,
  \qquad 0\leq E_r\leq C_{\chi_\infty}.
\]
Moreover, $R_q\in L^1\cap L^\infty$ implies
\[
  \sup_{0<\varepsilon\leq1/2}
  \|R_q\|_{L^{r_\varepsilon}(\R^6)}<\infty.
\]
Hence the reverse triangle inequality and \eqref{eq:q6} give
\[
  \varepsilon\left|
  \|\widehat{\Delta q}\|_{L^{r_\varepsilon}}
  -\|H_q\|_{L^{r_\varepsilon}}
  \right|
  \leq\varepsilon\|R_q\|_{L^{r_\varepsilon}}\longrightarrow0.
\]
Since $r_\varepsilon=1+\delta_\varepsilon$,
\[
  \varepsilon\delta_\varepsilon^{-1/r_\varepsilon}
  =\frac{\varepsilon}{\delta_\varepsilon}
  \delta_\varepsilon^{\delta_\varepsilon/(1+\delta_\varepsilon)}
  \longrightarrow6,
\]
where \eqref{eq:eps-delta-exact-expanded} gives $\varepsilon/\delta_\varepsilon\to6$ and $\delta_\varepsilon\log\delta_\varepsilon\to0$. The bounded term $E_{r_\varepsilon}$ disappears in the limit, and $A_{r_\varepsilon}\to A_1=2\pi^2/3$. Therefore
\[
  \lim_{\varepsilon\downarrow0}
  \varepsilon\|H_q\|_{L^{r_\varepsilon}(\R^6)}=768A_1,
\]
and consequently
\begin{equation*}
  \lim_{\varepsilon\downarrow0}
  \varepsilon(2\pi)^{-3}
  \|\widehat{\Delta q}\|_{L^{r_\varepsilon}(\R^6)}
  =(2\pi)^{-3}\,768\,\frac{2\pi^2}{3}=\frac{64}{\pi}.
\end{equation*}

Finally, \eqref{eq:dWiener} and \eqref{eq:q7} imply
\begin{equation}\label{eq:wiener-profile-expanded}
\begin{aligned}
  \mathfrak m_\varepsilon(d_J)
  &\leq\max_{\mathsf M\in\mathcal M(d_J)}
  C_{\mathsf M}(2\pi)^{-3/2}
  \|\widehat {d_J}\|_{L^1(\R^3)},\\
  \mathfrak m_\varepsilon(\partial_jq)
  &\leq\max_{\mathsf M\in\mathcal M(\partial_jq)}
  C_{\mathsf M}(2\pi)^{-3}
  \|\widehat{\partial_jq}\|_{L^1(\R^6)}.
\end{aligned}
\end{equation}
If $b(\mathsf M_0x-c_0)$ is an angular block and $c(\mathsf M_1x-c_1)$ is a Wiener block, Lemma~\ref{lem:structured-wiener}, followed by \eqref{eq:m-operator-expanded}, gives
\[
  \|b(\mathsf M_0\cdot-c_0)c(\mathsf M_1\cdot-c_1)f\|_{\FL_{-\varepsilon}^1}
  \leq\mathfrak m_{\varepsilon,\mathsf M_0}(b)
  C_{\mathsf M_1}(2\pi)^{-m_1/2}
  \|\widehat c\|_{L^1(\R^{m_1})}\|f\|_{\B^0}.
\]
Thus the product inherits only the angular pole, irrespective of the combined active dimension.

We can now define the five constants used in the statement
\begin{equation} \label{eq:profile-constants}
\begin{aligned}
  B_* & =\sup_{0<\varepsilon\leq1/2}
  \varepsilon\sum_{j=1}^3\mathfrak m_\varepsilon(b_j),\\
  H_* & =\sup_{0<\varepsilon\leq1/2}
  \varepsilon\mathfrak m_\varepsilon(\Delta q),              \\
  D_* & =\sup_{0<\varepsilon\leq1/2}
  \mathfrak m_\varepsilon(d_J),                         \\
  G_* & =
  \max_{\mathsf M\in\bigcup_{j=1}^6\mathcal M(\partial_jq)}
  C_{\mathsf M}(2\pi)^{-3}
  \sum_{j=1}^6\|\widehat{\partial_jq}\|_{L^1(\R^6)}.
\end{aligned}
\end{equation}
For each electron $i$, let
\[ \mathcal C_i=\{x\mapsto x_i-R_\nu:1\leq\nu\leq L\} \cup\{x\mapsto x_i-x_j:j\neq i\}. \]
For $\gamma,\gamma'\in\mathcal C_i$, Lemma~\ref{lem:two-angular-blocks} provides a surjective map $\mathsf M_{\gamma,\gamma'}:\R^{3N}\to\R^{m_{\gamma,\gamma'}}$ with $m_{\gamma,\gamma'}\in\{3,6\}$, and a profile $a_{\gamma,\gamma'}\in L_c^\infty(\R^{m_{\gamma,\gamma'}})$ such that
\[ b(\gamma(x))\cdot b(\gamma'(x)) =a_{\gamma,\gamma'}(\mathsf M_{\gamma,\gamma'}x). \]
Fix one such representation for each pair and define
\begin{equation}\label{eq:P-star}
  P_*=\sup_{0<\varepsilon\leq1/2}
  \max_{\substack{1\leq i\leq N\\ \gamma,\gamma'\in\mathcal C_i}}
  \varepsilon^2
  \mathfrak m_{\varepsilon,\mathsf M_{\gamma,\gamma'}}
  (a_{\gamma,\gamma'}).
\end{equation}
Equations \eqref{eq:b-m-explicit-expanded}--\eqref{eq:wiener-profile-expanded} prove that the five constants in \eqref{eq:profile-constants} and \eqref{eq:P-star} are finite and computable from the fixed cut-off profiles and the fixed collision configuration.

\subsection{Term-by-term estimate of the conjugated equations}

This subsection applies the five profile estimates to every term in the conjugated equations for $\phi$ and $\phi_3$. We first introduce a single coefficient count for the two-particle factor and a parameter that distinguishes the two equations. The drift terms, the full expansion of the squared gradient, and the remaining zeroth-order terms are then estimated separately and grouped according to whether they carry two, one, or no critical scales. Finally, a high-frequency contraction at the fixed loss $\varepsilon=1/2$, combined with a low-frequency Cauchy--Schwarz estimate, proves $\|u\|_{\B^1}\leq C\|u\|_{H^1}$ without introducing any additional dependence on $\varepsilon$.

Set
\begin{equation*} Z_\Sigma=\sum_{\nu=1}^LZ_\nu, \qquad c_F=\frac{Z_\Sigma}{2}+\frac{N-1}{4}, \end{equation*}
and let
\[ \tau=0\quad\text{for }u=\phi, \qquad \tau=\frac{\pi-2}{12\pi}\binom N2Z_\Sigma \quad\text{for }u=\phi_3. \]
Thus $\tau$ distinguishes the two conjugated equations and allows us to estimate them without duplicating the argument.
For a fixed electron, the sum of the absolute coefficients in $\nabla F$ is $c_F$; summing over all electrons gives $Nc_F$.
The sum of the absolute coefficients in $P$ is $\tau$ when that factor is present.

For each term of \eqref{eq:phi-equation} or \eqref{eq:phi3-equation}, we first apply Lemma~\ref{lem:structured-wiener} to all Wiener factors and then apply \eqref{eq:m-operator-expanded} to the single possible non-Wiener block. Since $\|\partial_j u\|_{\FL_0^1}\leq\|u\|_{\B^1}$, the two drift terms obey
  \begin{equation} \label{eq:drift-F-Q-expanded}
    \begin{aligned}
    \|2\nabla F\cdot\nabla u\|_{\FL_{-\varepsilon}^1}
   & \leq\frac{2Nc_FB_*}{\varepsilon}\|u\|_{\B^1}, \\
  \|2\nabla P\cdot\nabla u\|_{\FL_{-\varepsilon}^1}
   & \leq2\tau G_*\|u\|_{\B^1}. 
    \end{aligned}
  \end{equation}
Expanding the square and summing absolute coefficient products gives
  \begin{equation*}
    \begin{aligned}
  \||\nabla F|^2u\|_{\FL_{-\varepsilon}^1}
   & \leq\frac{Nc_F^2P_*}{\varepsilon^2}\|u\|_{\B^1}, \\
  \|2\nabla F\cdot\nabla P\,u\|_{\FL_{-\varepsilon}^1}
   & \leq\frac{2Nc_F\tau B_*G_*}{\varepsilon}\|u\|_{\B^1},  \\
  \||\nabla P|^2u\|_{\FL_{-\varepsilon}^1}
   & \leq\tau^2G_*^2\|u\|_{\B^1}.   
    \end{aligned}
  \end{equation*}
The remaining zeroth-order terms satisfy
  \begin{equation} \label{eq:remaining-zero-order}
    \begin{aligned}
    \|\Delta(F-F_2)u\|_{\FL_{-\varepsilon}^1}
   & \leq Nc_FD_*\|u\|_{\B^1}, \\
  \|(\Delta P)u\|_{\FL_{-\varepsilon}^1}
   & \leq\frac{\tau H_*}{\varepsilon}\|u\|_{\B^1},  \\
  \|(E+1)u\|_{\FL_{-\varepsilon}^1}
   & \leq|E+1|\|u\|_{\B^1}.   
    \end{aligned}
  \end{equation}
Grouping \eqref{eq:drift-F-Q-expanded}--\eqref{eq:remaining-zero-order} according to pole order gives
\begin{equation*}
\begin{aligned}
  M_2 & =Nc_F^2P_*, \\
  M_1 & =2Nc_FB_*+\tau H_*+2Nc_F\tau B_*G_*,          \\
  M_0 & =Nc_FD_*+|E+1|+2\tau G_*+\tau^2G_*^2.  
\end{aligned}  
\end{equation*}
For the fixed range $0<\varepsilon\leq1/2$, define
\begin{equation}\label{eq:M-total-expanded} M=M_2+\frac12M_1+\frac14M_0. \end{equation}

\begin{proof}[Proof of Theorem~\ref{thm:main}(ii)]
  Combining the preceding eight inequalities yields
  \[ \|(-\Delta+1)u\|_{\FL_{-\varepsilon}^1} \leq\left(\frac{M_2}{\varepsilon^2} +\frac{M_1}{\varepsilon}+M_0\right)\|u\|_{\B^1}. \]
Since $M_2\varepsilon^{-2}+M_1\varepsilon^{-1}+M_0 \leq M\varepsilon^{-2}$ for $0<\varepsilon\leq1/2$, applying \eqref{eq:resolvent-lift-expanded} with $p=1$ and $2\beta=\varepsilon$ gives
  \begin{equation*} \|u\|_{\B^{2-\varepsilon}} \leq\frac{M}{\varepsilon^2}\|u\|_{\B^1}. \end{equation*}

It remains to control $\|u\|_{\B^1}$ independently of $\varepsilon$. We reuse \eqref{eq:drift-F-Q-expanded}--\eqref{eq:remaining-zero-order} at the fixed loss $\varepsilon=1/2$. Since $\alpha_{1/2}=12(1+2^{-1/2})>12$, Lemma~\ref{lem:multiplier} applies to every possible non-Wiener block for both $p=1$ and $p=2$, while Lemma~\ref{lem:structured-wiener} controls all Wiener factors. Hence those estimates remain valid with $\FL^1$ replaced by $\FL^p$.
Let $\Lambda_0=4M_2+2M_1+M_0$.  For $u\in\{\phi,\phi_3\}$, let $\mathcal T_u f$ be the resolvent applied to the right-hand side of \eqref{eq:phi-equation} or \eqref{eq:phi3-equation}, respectively, with the unknown replaced by $f$.
Thus $u=\mathcal T_u u$.  The preceding coefficient estimates and \eqref{eq:resolvent-lift-expanded}, applied with $\beta=1/4$ and $p\in\{1,2\}$, give
\begin{equation*} \|\mathcal T_u f\|_{\FL_{3/2}^p}\leq\Lambda_0\|f\|_{\FL_1^p},\qquad u\in\{\phi,\phi_3\},\quad p=1,2. \end{equation*}
  Let $\Pi_K$ be the Fourier projection to $|\xi|>K$ and take $K=\max\{1,4\Lambda_0^2\}.$
  On $|\xi|>K$, $\ip\xi\leq\ip K^{-1/2}\ip\xi^{3/2}$, and hence
  \[ \|\Pi_K\mathcal T_u\|_{\FL_1^p\to\FL_1^p} \leq\Lambda_0\ip K^{-1/2}\leq\frac12, \qquad p=1,2. \]
  Write $u_{\rm lo}=(I-\Pi_K)u$ and $u_{\rm hi}=\Pi_Ku$. Since $u=\mathcal T_u u$,
  \[ (I-\Pi_K\mathcal T_u)u_{\rm hi} =\Pi_K\mathcal T_u u_{\rm lo}. \]
  The Neumann series for the inverse converges both in $H^1$ and in $\B^1$, and the two sums represent the same tempered distribution. Therefore
  \[ \|u_{\rm hi}\|_{\B^1} \leq2\|\Pi_K\mathcal T_u u_{\rm lo}\|_{\B^1} \leq\|u_{\rm lo}\|_{\B^1}. \]
  Since $\widehat u_{\rm lo}$ is supported in $B_{3N}(0,K)$, the Cauchy--Schwarz inequality gives
\[ \|u_{\rm lo}\|_{\B^1}\leq |B_{3N}(0,K)|^{1/2}\|u\|_{H^1} =\left(\frac{\pi^{3N/2}K^{3N}}{\Gamma(3N/2+1)}\right)^{1/2}\|u\|_{H^1}. \]
It follows that
\[ \|u\|_{\B^1}\leq\|u_{\rm lo}\|_{\B^1}+\|u_{\rm hi}\|_{\B^1} \leq2\left(\frac{\pi^{3N/2}K^{3N}}{\Gamma(3N/2+1)}\right)^{1/2}\|u\|_{H^1}. \]
  The right-hand side is independent of $\varepsilon$, which completes the proof.
\end{proof}

\section{Optimality among universal factors}\label{sec:one-electron}
\begin{proof}[Proof of Theorem~\ref{thm:main}(iii)]
If such a universal factor existed, the embedding $\B^2(\R^{3N})\hookrightarrow C^2(\R^{3N})$ \cite[Corollary~2.13(1), (2.20)]{LiaoMing2025} would give a state-independent factorization whose quotient is $C^2$ for every
eigenfunction. This is excluded by the optimality result of Fournais et al.~\cite[Theorem~1.1 and the proof of optimality]{FournaisEtAl2005}, which completes the proof. \end{proof}

\section{Sharpness for an unperturbed two-electron atom}\label{sec:sharpness}

Throughout this section, we specialize to $N=2$, $L=1$, and a nucleus of charge $Z>0$ at the origin.  For the bound state in Theorem~\ref{thm:main}(iv), we write $u=\phi_3$.  Define, almost everywhere on $\R^6$,
\begin{equation*} \Gamma(x,y)= \left(\frac{x}{|x|}-\frac{y}{|y|}\right)\cdot\frac{x-y}{|x-y|}, \qquad Q(x,y)=\frac{x\cdot y}{|x|^2+|y|^2}, \qquad \Gamma_0=\Gamma-\frac{16(2-\pi)}{3\pi}Q. \end{equation*}
The values on the three collision sets may be chosen arbitrarily, because these sets have six-dimensional measure zero.  The coefficient in $\Gamma_0$ is the one for which Fournais et al.~construct a homogeneous $C^{1,1}$ second-order correction whose Laplacian equals $\Gamma_0$; see \cite[Lemma~A.1, (A.1)--(A.4)]{FournaisEtAl2005}.  Only its Laplacian $\Gamma_0$ is needed below.

This section proves Theorem~\ref{thm:main}(iv) by isolating the unique term in the two-electron conjugated equation that carries two independent critical frequency scales. We first compute the exact double-pole residue of the localized homogeneous source $\Gamma_0$. We then derive the local conjugated equation, in which $-(Z/4)\Gamma_0u$ is separated from the drift and lower-order terms. Replacing $u$ in this source by $u(0,0)$ leaves the double-pole coefficient unchanged, because one frequency derivative bounds the resulting difference by $C/\varepsilon$. A local coefficient classification and an exterior partition show that every remaining term also contributes at most $C/\varepsilon$. Combining these estimates yields \eqref{eq:sharp-main}; positivity at the triple coalescence then shows that the quadratic rate is attained by a two-electron ground state.

With $\varepsilon=2-s$ and Fourier variables $(\xi,\eta)\in\R^3\times\R^3$ dual to $(x,y)$, the quadratic divergence as $s\uparrow2$ arises from the two nested-frequency cones
\[ 1\ll|\eta|\ll|\xi|, \qquad 1\ll|\xi|\ll|\eta|. \]
On each cone, the leading homogeneous term gives two nested logarithmic radial integrations and hence a contribution of order $\varepsilon^{-2}$. The exact coefficient in Part~(iv) therefore records a genuine nested-frequency effect rather than an artifact of the multiplier estimate. The proof of Part~(iv) uses the exact conjugated equation and does not require convergence of a Fock expansion.

  \subsection{The localized three-particle source residue}

  The aim of this subsection is to prove Proposition~\ref{prop:source-residue}. We first express the localized source in Fourier space as a convolution of cut-off angular profiles together with the $Q$ correction. On the two nested frequency cones, this expression has an explicit homogeneous leading term. The localization error, the pointwise cone error, and the contribution from the complementary region are each shown to be $O(\varepsilon^{-1})$. Exact angular and radial integration of the leading term then gives the coefficient $128\pi$ in \eqref{eq:source-residue}, which supplies the doubly critical contribution used in the proof of Theorem~\ref{thm:main}(iv).

  \begin{proposition}[Localized source residue]\label{prop:source-residue}
    Let $\zeta\in C_c^\infty(\R^6)$ equal one near the origin.  There is $C_{\zeta}<\infty$ such that
    \begin{equation}\label{eq:source-residue} \left| \|\res(\zeta\Gamma_0)\|_{\B^{2-\varepsilon}(\R^6)} -\frac{128\pi}{\varepsilon^2} \right|\leq\frac{C_{\zeta}}{\varepsilon}, \qquad 0<\varepsilon\leq\frac12. \end{equation}
  \end{proposition}

  The following lemma collects the localization facts used both in the proof of Proposition~\ref{prop:source-residue} and in the subsequent freezing estimate.

  \begin{lemma}[Localized angular source calculus]
    \label{lem:localized-angular-source} Let $\rho\in C_c^\infty(\R^6)$ equal one near the origin.  There is a radial $\chi_\rho\in C_c^\infty(\R^3)$ such that, with
    \[ a_\rho(z)=\chi_\rho(z)\frac{z}{|z|}\quad(z\neq0), \qquad k_\rho=(2\pi)^{-3}\widehat\rho, \]
    and
    \[ \mathcal A_\rho(\xi,\eta) =\widehat a_\rho(\xi+\eta)\cdot \{\widehat a_\rho(-\eta)-\widehat a_\rho(\xi)\}, \]
    one has
    \begin{equation} \label{eq:localized-source-representation} \widehat{\rho\Gamma_0} =k_\rho*\mathcal A_\rho -\frac{16(2-\pi)}{3\pi}\widehat{\rho Q}, \qquad \int_{\R^6}k_\rho(\Theta)\dd\Theta=1. \end{equation}
    The radial choice of $\chi_\rho$ also gives
    \begin{equation} \label{eq:localized-F-symmetry} a_\rho(-z)=-a_\rho(z),\qquad \widehat a_\rho(-k)=-\widehat a_\rho(k),\qquad \mathcal A_\rho(\xi,\eta)=\mathcal A_\rho(\eta,\xi). \end{equation}
    The angular estimates \eqref{eq:general-angular-tail} and \eqref{eq:general-angular-global} hold with $\theta=\chi_\rho$ and $a_\theta=a_\rho$.
    If $\Xi=(\xi,\eta)$, then
    \begin{align}
      \left|\widehat{\rho Q}(\xi,\eta)
        +48\frac{\xi\cdot\eta}{(|\xi|^2+|\eta|^2)^4}\right|
      &\leq C_\rho|\Xi|^{-7},
      &&|\Xi|\geq1,                                      \label{eq:localized-Q-tail}\\
      |\nabla\widehat{\rho Q}(\Xi)|
      &\leq C_\rho\ip\Xi^{-7},
      &&\Xi\in\R^6.                                     \label{eq:localized-Q-gradient}
    \end{align}
    Finally,
    \begin{equation} \label{eq:localized-F-derivative} \int_{\R^6}\ip\Xi^{-\varepsilon} |\nabla \mathcal A_\rho(\Xi)|\dd\Xi \leq\frac{C_\rho}{\varepsilon}, \qquad 0<\varepsilon\leq\frac12. \end{equation}
  \end{lemma}

  \begin{proof}
    The images of $\supp\rho$ under $(x,y)\mapsto x,y,x-y$ have compact union.  Choose $\chi_\rho$ radial and equal to one on a ball containing this union.  Then, almost everywhere,
    \[ \rho\Gamma =\rho\,(a_\rho(x)-a_\rho(y))\cdot a_\rho(x-y). \]
    Both sides lie in $L^1(\R^6)$, so the identity also holds in $\mathcal S'(\R^6)$.  The determinant-one changes of variables $(z,w)=(x,x-y)$ and $(z,w)=(y,x-y)$, followed by the unitary product formula, give \eqref{eq:localized-source-representation}.  Fourier inversion gives $\int_{\R^6}k_\rho(\Theta)\dd\Theta=\rho(0)=1$.

    Since $\chi_\rho$ is radial, $a_\rho$ is odd and hence so is its Fourier transform.  Therefore
    \[ \mathcal A_\rho(\eta,\xi) =\widehat a_\rho(\xi+\eta)\cdot \{\widehat a_\rho(-\xi)-\widehat a_\rho(\eta)\} =\mathcal A_\rho(\xi,\eta), \]
    which proves \eqref{eq:localized-F-symmetry}.

    Put $s=|\xi|^2+|\eta|^2$.  In six dimensions the Riesz identity is $\widehat{|X|^{-2}}(\Xi)=2s^{-2}$.  Since $Q=(x\cdot y)|X|^{-2}$ and multiplication by $x_jy_j$ becomes $-\partial_{\xi_j}\partial_{\eta_j}$, direct differentiation gives
    \[ -\partial_{\xi_j}\partial_{\eta_j}(2s^{-2}) =-48\xi_j\eta_j s^{-4}. \]
    Summing over $j=1,2,3$ yields
    \[ \widehat Q(\xi,\eta) =-48\frac{\xi\cdot\eta}{(|\xi|^2+|\eta|^2)^4}, \qquad \Xi\neq0. \]
    Applying Lemma~\ref{lem:homogeneous-localization} to the degree-zero profile $Q$ with $L=7$ proves \eqref{eq:localized-Q-tail}.  For each component, $\partial_{\Xi_\ell}\widehat{\rho Q}=-i\widehat{X_\ell\rho Q}$.
    The profile $X_\ell Q$ is homogeneous of degree one; the localization lemma with $L=8$ shows that its localized transform differs from its degree $-7$ homogeneous transform by at most $C_\rho|\Xi|^{-8}$.  Since $X_\ell\rho Q\in L^1$, the transform is also bounded near the origin.
    This proves \eqref{eq:localized-Q-gradient}.

    It remains to prove \eqref{eq:localized-F-derivative}.  Under either of the determinant-one linear changes
    \[ (p,q)=(\xi+\eta,-\eta), \qquad (p,q)=(\xi+\eta,\xi), \]
    each summand of $\mathcal A_\rho$ is the bilinear dot product of $\widehat a_\rho(p)$ and $\widehat a_\rho(q)$, and the Japanese brackets are uniformly comparable.  By \eqref{eq:general-angular-global}, its gradient is bounded by
    \[ C_\rho\bigl(\ip p^{-4}\ip q^{-3} +\ip p^{-3}\ip q^{-4}\bigr). \]
    On $|p|\leq|q|$, the second term is bounded by the first and $\ip{(p,q)}^{-\varepsilon}\leq\ip q^{-\varepsilon}$; hence
    \[ \int_{\{(p,q)\in\R^6:\,|p|\leq|q|\}} \ip{(p,q)}^{-\varepsilon} |\nabla \mathcal A_\rho|\dd p\dd q \leq C_\rho\!\int_{\R^3}\ip p^{-4}\dd p \int_{\R^3}\ip q^{-3-\varepsilon}\dd q \leq\frac{C_\rho}{\varepsilon}. \]
    The region $|q|\leq|p|$ is symmetric.  This proves \eqref{eq:localized-F-derivative}.
  \end{proof}

  \begin{proof}[Proof of Proposition~\ref{prop:source-residue}]
Write $\Xi=(\xi,\eta)$, $R=|\xi|$, and $r=|\eta|$. Whenever $Rr>0$, set $\omega=\xi/R$ and $\nu=\eta/r$.
Set
\[ D_4^+=\{R\geq4r\geq16\}, \qquad D_4^-=\{r\geq4R\geq16\}, \qquad \Omega_4=D_4^+\cup D_4^-. \]
The two sets in \(\Omega_4\) are the nested frequency cones on which the two radial variables can generate independent divergences.  This frequency geometry is summarized in Figure~\ref{fig:frequency-region-decomposition}.
Panel~(a) displays the cones and the three types of regions used to cover their complement.  Panel~(b) records the dyadic consequence used in Step~4.

\begin{figure}[t]
  \centering
  \begin{tikzpicture}[
      x=0.82cm,
      y=0.82cm,
      >=Latex,
      axis/.style={->,line width=0.75pt},
      boundary/.style={line width=0.75pt},
      guide/.style={densely dashed,gray!75,line width=0.45pt},
      threshold/.style={densely dotted,gray!75,line width=0.45pt},
      region label/.style={font=\scriptsize,align=center},
      panel label/.style={font=\small\bfseries,anchor=west}
    ]

            \begin{scope}
      \def\xmax{7.0}
      \def\ymax{7.0}
      \def\four{1.25}
      \def\eight{2.50}
      \def\sixteen{5.00}

      \node[panel label] at (0,7.75) {(a) Regions in the $(r,R)$-plane};

                  \fill[gray!12] (0,0) rectangle (\xmax,\ymax);
      \fill[orange!16] (0,0) rectangle (\eight,\eight);
      \fill[green!15]  (0,\eight) rectangle (\four,\ymax);
      \fill[yellow!22] (\eight,0) rectangle (\xmax,\four);
      \fill[red!18]
        (\four,\sixteen) -- (1.75,\ymax) -- (\four,\ymax) -- cycle;
      \fill[blue!17]
        (\sixteen,\four) -- (\xmax,1.75) -- (\xmax,\four) -- cycle;

            \draw[boundary,red!70!black]
        (\four,\sixteen) -- (1.75,\ymax);
      \draw[boundary,blue!70!black]
        (\sixteen,\four) -- (\xmax,1.75);
      \draw[guide] (\four,0) -- (\four,\ymax);
      \draw[guide] (0,\four) -- (\xmax,\four);
      \draw[threshold] (\eight,0) -- (\eight,\eight);
      \draw[threshold] (0,\eight) -- (\eight,\eight);
      \draw[guide] (\four,\sixteen) -- (0,\sixteen);
      \draw[guide] (\sixteen,\four) -- (\sixteen,0);

            \draw[axis] (0,0) -- (\xmax+0.35,0) node[below] {$r=|\eta|$};
      \draw[axis] (0,0) -- (0,\ymax+0.35) node[left] {$R=|\xi|$};
      \foreach \x/\lab in {\four/4,\eight/8,\sixteen/16}
        \draw (\x,0.08) -- (\x,-0.08)
          node[below=2pt,font=\scriptsize] {$\lab$};
      \foreach \y/\lab in {\four/4,\eight/8,\sixteen/16}
        \draw (0.08,\y) -- (-0.08,\y)
          node[left=2pt,font=\scriptsize] {$\lab$};

            \node[region label,text=red!65!black,anchor=west] (pluslabel)
        at (2.05,6.45) {Step 2\\[-1pt] $D_4^+$\\[-1pt]
          $R\geq4r$, $r\geq4$};
      \draw[-Latex,red!65!black,thin]
        (pluslabel.west) -- (1.50,6.12);
      \node[region label,text=blue!65!black,anchor=west] (minuslabel)
        at (4.05,2.28) {Step 2\\[-1pt] $D_4^-$\\[-1pt]
          $r\geq4R$, $R\geq4$};
      \draw[-Latex,blue!65!black,thin,shorten <=1pt,shorten >=1pt]
        (minuslabel.south) -- (5.92,1.48);
      \node[region label] at (4.15,4.10)
        {high--high complement\\[-1pt] $R/4<r<4R$};
      \node[region label,rotate=90] at (0.61,4.05)
        {one-radius strip\\[-1pt] $r<4$, $R\geq8$};
      \node[region label] at (4.70,0.58)
        {one-radius strip\\[-1pt] $R<4$, $r\geq8$};
      \node[region label] at (1.40,1.52)
        {bounded core\\[-1pt] $R,r<8$};

      \fill[red!70!black]  (\four,\sixteen) circle (1.25pt);
      \fill[blue!70!black] (\sixteen,\four) circle (1.25pt);
    \end{scope}

        \begin{scope}[shift={(8.55,0.35)}]
      \def\amax{6.4}
      \def\jmin{0.90}
      \def\band{1.35}

      \node[panel label] at (0,7.40) {(b) Dyadic shell indices};

      \begin{scope}
        \clip (\jmin,\jmin) rectangle (\amax,\amax);
        \fill[gray!14]
          (\jmin,\jmin) -- (\amax,\jmin) -- (\amax,\amax)
          -- (\jmin,\amax) -- cycle;
        \fill[red!18]
          (\jmin,\jmin+\band) -- (\amax-\band,\amax)
          -- (\jmin,\amax) -- cycle;
        \fill[blue!17]
          (\jmin+\band,\jmin) -- (\amax,\amax-\band)
          -- (\amax,\jmin) -- cycle;
      \end{scope}

      \draw[guide] (\jmin,0) -- (\jmin,\amax);
      \draw[guide] (0,\jmin) -- (\amax,\jmin);
      \draw[boundary,red!70!black]
        (\jmin,\jmin+\band) -- (\amax-\band,\amax)
        node[pos=0.82,above left=-1pt,font=\scriptsize] {$j-k=2$};
      \draw[boundary,blue!70!black]
        (\jmin+\band,\jmin) -- (\amax,\amax-\band)
        node[pos=0.82,below right=-1pt,font=\scriptsize] {$k-j=2$};
      \draw[densely dashed,gray!60]
        (\jmin,\jmin) -- (\amax,\amax);

            \foreach \x in {1.35,2.05,...,5.55}
        \foreach \y in {1.35,2.05,...,5.55}
          {
            \pgfmathparse{abs(\x-\y) < 1.20 ? 1 : 0}
            \ifnum\pgfmathresult=1
              \fill[gray!65] (\x,\y) circle (0.85pt);
            \fi
          }
      \draw[very thick,black!70,rounded corners=1pt]
        (3.80,3.80) rectangle (4.48,4.48);
      \node[region label,anchor=west] at (4.58,4.15)
        {$\mathcal S_{j,k}$};

      \draw[axis] (0,0) -- (\amax+0.35,0)
        node[below] {$k$ for $r\sim2^k$};
      \draw[axis] (0,0) -- (0,\amax+0.35)
        node[left] {$j$ for $R\sim2^j$};
      \node[font=\scriptsize,below] at (\jmin,0) {$2$};
      \node[font=\scriptsize,left] at (0,\jmin) {$2$};

      \node[region label,text=red!65!black] at (2.00,5.45) {$D_4^+$};
      \node[region label,text=blue!65!black] at (5.45,2.00) {$D_4^-$};
      \node[region label,fill=white,fill opacity=0.82,text opacity=1,
        inner sep=2pt] at (3.35,3.00)
        {$|j-k|\leq2$\\[-1pt] one unbounded index};
    \end{scope}
  \end{tikzpicture}
  \caption{Frequency-region decomposition used in Steps~2 and~4 of the proof of Proposition~\ref{prop:source-residue}.  Panel~(a) shows the two nested cones and the high--high, one-radius, and bounded parts of their complement.
  The threshold $8$ only separates the one-radius strips from the bounded core; it is not part of the definition of $D_4^\pm$.  In Panel~(b), if $\mathcal S_{j,k}\subset\Omega_4^c$ is nonempty with $j,k\geq2$, then $|j-k|\leq2$.  Thus, for each $n=\max\{j,k\}$, at most five shell pairs occur, and only one dyadic index is independently unbounded.  Boundary sets may be assigned to either adjacent region.}
  \label{fig:frequency-region-decomposition}
\end{figure}
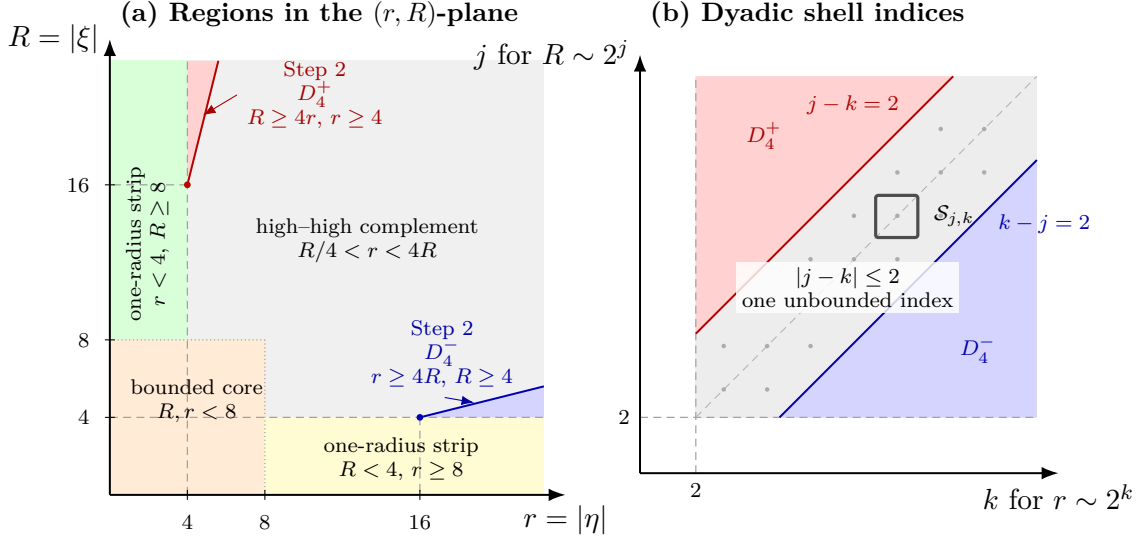

The corresponding homogeneous leading term is
\[ \mathcal H(\xi,\eta) := \frac8\pi(\omega\cdot\nu)R^{-3}r^{-3}. \]

Apply Lemma~\ref{lem:localized-angular-source} with $\rho=\zeta$, and fix the corresponding $a_\zeta$, $k_\zeta$, and $\mathcal A_\zeta$.  Then
\[ \widehat{\zeta\Gamma_0} = k_\zeta*\mathcal A_\zeta -\frac{16(2-\pi)}{3\pi}\widehat{\zeta Q}. \]
In particular, \eqref{eq:localized-F-symmetry} gives $\mathcal A_\zeta(\xi,\eta)=\mathcal A_\zeta(\eta,\xi)$; this identity justifies every interchange of the two frequency variables below.
Applying \eqref{eq:resolvent-lift-expanded} with $p=1$ and $2\beta=\varepsilon$ to the representation above gives
\[ \|\res(\zeta\Gamma_0)\|_{\B^{2-\varepsilon}} = \int_{\R^6} \ip\Xi^{-\varepsilon} \left| k_\zeta*\mathcal A_\zeta(\Xi) -\frac{16(2-\pi)}{3\pi}\widehat{\zeta Q}(\Xi) \right|\dd\Xi. \]

The reverse triangle inequality, applied separately on \(\Omega_4\) and its complement, gives the global decomposition
\[
\begin{aligned}
  &\left|
    \|\res(\zeta\Gamma_0)\|_{\B^{2-\varepsilon}}
    -
    \int_{\Omega_4}\ip\Xi^{-\varepsilon}|\mathcal H(\Xi)|\dd\Xi
  \right|                                                     \\
  &\quad\leq
  \int_{\R^6}
  \ip\Xi^{-\varepsilon}
  |k_\zeta*\mathcal A_\zeta-\mathcal A_\zeta|(\Xi)\dd\Xi                      \\
  &\qquad+
  \int_{\Omega_4}
  \ip\Xi^{-\varepsilon}
  \left(
    |\mathcal A_\zeta(\Xi)-\mathcal H(\Xi)|
    +\frac{16(\pi-2)}{3\pi}|\widehat{\zeta Q}(\Xi)|
  \right)\dd\Xi                                               \\
  &\qquad+
  \int_{\Omega_4^c}
  \ip\Xi^{-\varepsilon}
  \left(
    |\mathcal A_\zeta(\Xi)|
    +\frac{16(\pi-2)}{3\pi}|\widehat{\zeta Q}(\Xi)|
  \right)\dd\Xi .
\end{aligned}
\]
Thus it remains to establish the following four estimates.
\begin{align}
  \int_{\R^6}\ip\Xi^{-\varepsilon}
  |k_\zeta*\mathcal A_\zeta-\mathcal A_\zeta|(\Xi)\dd\Xi
  &\leq \frac{C_\zeta}{\varepsilon},
  \label{eq:source-cutoff-error}                                       \\
  \int_{\Omega_4}\ip\Xi^{-\varepsilon}
  \bigl(|\mathcal A_\zeta-\mathcal H|+|\widehat{\zeta Q}|\bigr)\dd\Xi
  &\leq \frac{C_\zeta}{\varepsilon},
  \label{eq:source-cone-error}                                         \\
  \int_{\Omega_4^c}\ip\Xi^{-\varepsilon}
  \bigl(|\mathcal A_\zeta|+|\widehat{\zeta Q}|\bigr)\dd\Xi
  &\leq \frac{C_\zeta}{\varepsilon},
  \label{eq:source-complement-error}                                   \\
  \left|
  \int_{\Omega_4}\ip\Xi^{-\varepsilon}|\mathcal H(\Xi)|\dd\Xi
  -\frac{128\pi}{\varepsilon^2}
  \right|
  &\leq \frac{C}{\varepsilon}.
  \label{eq:source-leading-pole}
\end{align}
Step~1 supplies the Fourier representation.  Steps~2--5 prove, respectively, \eqref{eq:source-cone-error}, \eqref{eq:source-cutoff-error}, \eqref{eq:source-complement-error}, and \eqref{eq:source-leading-pole}.

    \emph{Step 1. Shared localized-source input.}
    Lemma~\ref{lem:localized-angular-source} proves the representation above and supplies all localization estimates used below; no profile is relocalized inside this proposition.

    Since $\zeta\Gamma_0\in L^1$, the unitary Fourier normalization gives $\|\widehat{\zeta\Gamma_0}\|_{L^\infty(\R^6)} \leq(2\pi)^{-3}\|\zeta\Gamma_0\|_{L^1(\R^6)}$. Hence, uniformly for $0<\varepsilon\leq1/2$,
    \[ \int_{|\Xi|\leq1}\ip\Xi^{-\varepsilon} |\widehat{\zeta\Gamma_0}(\Xi)|\dd\Xi \leq \frac{|\mathbb S^5|}{6(2\pi)^3} \|\zeta\Gamma_0\|_{L^1(\R^6)}. \]
    Thus all endpoint divergence is a high-frequency phenomenon.

\emph{Step 2. Asymptotics on the two nested cones.}
Set $p=\xi+\eta$ and
\[ e_\zeta(k):=\widehat a_\zeta(k) +2i\sqrt{\frac2\pi}\frac{k}{|k|^4}. \]
On $D_4^+$,
\[ \frac34R\leq|p|\leq\frac54R, \qquad |e_\zeta(p)|+|e_\zeta(\xi)|\leq C_\zeta R^{-4}, \qquad |e_\zeta(-\eta)|\leq C_\zeta r^{-4}. \]
For $k\neq0$,
\[ D\!\left(\frac{k}{|k|^4}\right)h =|k|^{-4}h-4|k|^{-6}(k\cdot h)k, \qquad \left\|D\!\left(\frac{k}{|k|^4}\right)\right\|_{\mathrm{op}} =3|k|^{-4}. \]
Since $|\xi+t\eta|\geq3R/4$ for $0\leq t\leq1$, the mean-value theorem gives
\[ \left| \frac{p}{|p|^4}-\frac{\xi}{R^4} \right| \leq \frac{256}{27}rR^{-4}. \]
Using the definition of $\mathcal A_\zeta$, direct expansion yields
\[
\begin{aligned}
  \mathcal A_\zeta-\mathcal H
  ={}&\frac8\pi
  \left(\frac{p}{|p|^4}-\frac{\xi}{R^4}\right)
  \cdot\frac{\eta}{r^4}
  +\frac8\pi\frac{p\cdot\xi}{|p|^4R^4}                         \\
  &+2i\sqrt{\frac2\pi}\,e_\zeta(p)\cdot
  \left(\frac{\eta}{r^4}+\frac{\xi}{R^4}\right)
  +\widehat a_\zeta(p)\cdot
  \bigl(e_\zeta(-\eta)-e_\zeta(\xi)\bigr).
\end{aligned}
\]
The four terms are bounded, respectively, by
\[ C_\zeta R^{-4}r^{-2},\qquad C_\zeta R^{-6},\qquad C_\zeta(R^{-4}r^{-3}+R^{-7}),\qquad C_\zeta(R^{-3}r^{-4}+R^{-7}). \]
Because $r\leq R$ and $r\geq4$,
\[ R^{-6}\leq R^{-4}r^{-2},\qquad R^{-4}r^{-3}+R^{-7}\leq C R^{-3}r^{-4}. \]
Consequently,
\[ |\mathcal A_\zeta-\mathcal H| \leq C_\zeta\bigl(R^{-4}r^{-2}+R^{-3}r^{-4}\bigr) \qquad\text{on }D_4^+. \]
After multiplication by the polar measure $R^2r^2\dd R\dd r$ and using $\ip\Xi^{-\varepsilon}\leq R^{-\varepsilon}$, the two radial integrals are equal.
\begin{align*}
  \int_{16}^\infty R^{-2-\varepsilon}\int_4^{R/4}\dd r\dd R
  &=\frac{4^{-1-2\varepsilon}}{\varepsilon(1+\varepsilon)},\\
  \int_{16}^\infty R^{-1-\varepsilon}
       \int_4^{R/4}\frac{\dd r}{r^2}\dd R
  &=\frac{4^{-1-2\varepsilon}}{\varepsilon(1+\varepsilon)}.
\end{align*}
Thus the $\mathcal A_\zeta-\mathcal H$ contribution on $D_4^+$ is at most $C_\zeta/\varepsilon$.  Using \eqref{eq:localized-F-symmetry} and interchanging $(\xi,R)$ with $(\eta,r)$ gives
\[ |\mathcal A_\zeta-\mathcal H| \leq C_\zeta\bigl(R^{-2}r^{-4}+R^{-4}r^{-3}\bigr) \qquad\text{on }D_4^-, \]
and the same integral bound.

Finally, \eqref{eq:localized-Q-tail} and $r\geq4$ imply $|\widehat{\zeta Q}(\xi,\eta)|\leq C_\zeta R^{-7}r$ on $D_4^+$.
    Hence its contribution is bounded by
    \[ C_\zeta\int_{16}^\infty R^{-5-\varepsilon} \int_4^{R/4}r^3\dd r\dd R \leq\frac{C_\zeta4^{-4-2\varepsilon}}{4\varepsilon}. \]
    On $D_4^-$ the interchanged estimate is $|\widehat{\zeta Q}(\xi,\eta)|\leq C_\zeta r^{-7}R$, and therefore
    \[ \int_{D_4^-}\ip\Xi^{-\varepsilon} |\widehat{\zeta Q}(\Xi)|\dd\Xi \leq C_\zeta\int_{16}^\infty r^{-5-\varepsilon} \int_4^{r/4}R^3\dd R\dd r \leq\frac{C_\zeta4^{-4-2\varepsilon}}{4\varepsilon}. \]
    Combining the estimates on $D_4^+$ and $D_4^-$ proves \eqref{eq:source-cone-error}.

    \emph{Step 3. Error produced by the physical cut-off.}
    We prove \eqref{eq:source-cutoff-error}.
    Lemma~\ref{lem:localized-angular-source} gives \eqref{eq:localized-F-derivative} with $\rho=\zeta$.  Moreover, $\int_{\R^6}k_\zeta(\Theta)\dd\Theta=1$ and $k_\zeta$ is Schwartz.
    Therefore \eqref{eq:weighted-localization} bounds the left-hand side of \eqref{eq:source-cutoff-error} by
    \[ \frac{2^{\varepsilon/2}C_\zeta}{\varepsilon} \int_{\R^6} |\Theta|\ip\Theta^{1/2}|k_\zeta(\Theta)|\dd\Theta \leq\frac{C_\zeta}{\varepsilon}. \]
    This proves \eqref{eq:source-cutoff-error}.

    \emph{Step 4. The complement has only one unbounded scale.}
    We prove \eqref{eq:source-complement-error}.
    We use the covering
    \[ \Omega_4^c\subset \{R,r\geq4\} \cup\{r<4,\ R\geq8\} \cup\{R<4,\ r\geq8\} \cup\{R,r<8\}. \]
    All constants below are independent of the dyadic indices and of \(0<\varepsilon\leq1/2\).

    Suppose first that \(R,r\geq4\).  Since \(\Xi\notin D_4^+\cup D_4^-\), the definitions of the cones imply
    \[ \frac R4<r<4R. \]
    The region \(R,r\geq4\) is covered by dyadic shells with indices \(j,k\geq2\).  For such integers, define
    \[ \mathcal S_{j,k}:=\Omega_4^c\cap \left\{2^j\leq R<2^{j+1},\quad 2^k\leq r<2^{k+1}\right\}. \]
    If \(\mathcal S_{j,k}\neq\varnothing\), then
    \[ 2^j\leq R<4r<2^{k+3}, \qquad 2^k\leq r<4R<2^{j+3}. \]
    Thus \(j-k<3\) and \(k-j<3\).  Since \(j-k\) is an integer,
    \[ |j-k|\leq2. \]
    In particular, if \(n=\max\{j,k\}\), then
    \[ 2^{n-2}\leq R,r<2^{n+1} \qquad\text{on }\mathcal S_{j,k}. \]

    Fix a nonempty shell \(\mathcal S_{j,k}\), and put \(p=\xi+\eta\) and \(q=\eta\).
    This change of variables has Jacobian one, \(p-q=\xi\), and
    \[ \mathcal A_\zeta(\xi,\eta) =\widehat a_\zeta(p)\cdot \bigl(\widehat a_\zeta(-q)-\widehat a_\zeta(p-q)\bigr). \]
    We split \(\mathcal S_{j,k}\) into the regions \(|p|\leq|q|/2\) and \(|p|>|q|/2\).

    On the first region, for \(0\leq t\leq1\),
    \[ |-q+tp|\geq|q|-t|p|\geq\frac{|q|}{2} \geq2^{n-3}. \]
    The fundamental theorem of calculus and \eqref{eq:general-angular-global} therefore give
    \begin{align*}
      \bigl|\widehat a_\zeta(-q)-\widehat a_\zeta(p-q)\bigr|
      &=\left|\int_0^1
        p\cdot\nabla\widehat a_\zeta(-q+tp)\dd t\right|\\
      &\leq C_\zeta|p|2^{-4n}.
    \end{align*}
    Since \(|\widehat a_\zeta(p)|\leq C_\zeta\ip p^{-3}\), enlarging the transformed integration domain gives
    \begin{align*}
      &\int_{\mathcal S_{j,k}\cap\{|p|\leq|q|/2\}}
      |\mathcal A_\zeta(\xi,\eta)|\dd\xi\dd\eta\\
      &\quad\leq
      C_\zeta2^{-4n}
      \int_{2^{n-2}\leq|q|<2^{n+1}}\dd q
      \int_{|p|\leq2^n}|p|\ip p^{-3}\dd p\\
      &\quad\leq
      C_\zeta2^{-4n}2^{3n}
      \int_0^{2^n}\frac{t^3}{(1+t^2)^{3/2}}\dd t
      \leq C_\zeta.
    \end{align*}
    The last inequality follows from \(t^3(1+t^2)^{-3/2}\leq1\), so the radial integral is at most \(2^n\).  This difference estimate is the cancellation needed near the frequency hyperplane \(p=0\).

    On the remaining region, \(|q|\) and \(|p-q|\) both lie between \(2^{n-2}\) and \(2^{n+1}\).  Moreover,
    \[ |p|>\frac{|q|}{2}\geq2^{n-3}, \qquad |p|\leq|q|+|p-q|<2^{n+2}. \]
    Thus \(|p|\asymp|q|\asymp|p-q|\asymp2^n\), with absolute comparison constants, and the global decay of \(\widehat a_\zeta\) gives
    \[ |\mathcal A_\zeta(\xi,\eta)|\leq C_\zeta2^{-6n}. \]
    The shell \(\mathcal S_{j,k}\) is contained in a product of two three-dimensional annuli.  Hence
    \[ |\mathcal S_{j,k}| \leq \left(\frac{4\pi}{3}2^{3(j+1)}\right) \left(\frac{4\pi}{3}2^{3(k+1)}\right) \leq C2^{6n}. \]
    It follows that
    \[ \int_{\mathcal S_{j,k}\cap\{|p|>|q|/2\}} |\mathcal A_\zeta(\xi,\eta)|\dd\xi\dd\eta \leq C_\zeta. \]
    Combining the two regions, we obtain
    \[ \int_{\mathcal S_{j,k}}|\mathcal A_\zeta(\Xi)|\dd\Xi \leq C_\zeta. \]

    The \(Q\)-term satisfies the same shell bound.  Indeed, \eqref{eq:localized-Q-tail} and \(R,r,|\Xi|\asymp2^n\) imply
    \begin{align*}
      |\widehat{\zeta Q}(\xi,\eta)|
      &\leq48\frac{Rr}{(R^2+r^2)^4}+C_\zeta|\Xi|^{-7}\\
      &\leq C_\zeta2^{-6n}.
    \end{align*}
    Therefore
    \[ \int_{\mathcal S_{j,k}}|\widehat{\zeta Q}(\Xi)|\dd\Xi \leq C_\zeta. \]

    For each fixed \(n\), at most five pairs \((j,k)\) satisfy \(\max\{j,k\}=n\) and \(|j-k|\leq2\).  Since \(\ip\Xi^{-\varepsilon}\leq2^{-n\varepsilon}\) on every such shell, summation over the high--high part of the complement gives
    \begin{align*}
      &\int_{\Omega_4^c\cap\{R,r\geq4\}}
      \ip\Xi^{-\varepsilon}
      \bigl(|\mathcal A_\zeta(\Xi)|+|\widehat{\zeta Q}(\Xi)|\bigr)\dd\Xi\\
      &\qquad\leq
      5C_\zeta
      \sum_{n=2}^\infty2^{-n\varepsilon}
      \leq\frac{C_\zeta}{\varepsilon}.
    \end{align*}
    Here we used \(1-2^{-\varepsilon}\geq(\varepsilon\log2)/2\) for \(0<\varepsilon\leq1/2\).

    We next treat a region in which only one radius is unbounded.  Suppose \(r<4\) and \(R\geq8\).  Then
    \[ |\xi+\eta|\geq R-r\geq\frac R2. \]
    Since \(\widehat a_\zeta\) is bounded and \(|\widehat a_\zeta(k)|\leq C_\zeta\ip k^{-3}\), the definition of \(\mathcal A_\zeta\) gives
    \[ |\mathcal A_\zeta(\xi,\eta)| \leq C_\zeta R^{-3} \bigl(1+R^{-3}\bigr) \leq C_\zeta R^{-3}. \]
    Moreover, \eqref{eq:localized-Q-tail}, \(r<4\), and \(R\geq8\) imply
    \[ |\widehat{\zeta Q}(\xi,\eta)| \leq48\frac{Rr}{(R^2+r^2)^4}+C_\zeta|\Xi|^{-7} \leq192R^{-7}+C_\zeta R^{-7} \leq C_\zeta R^{-7}. \]
    Since \(\ip\Xi^{-\varepsilon}\leq R^{-\varepsilon}\), polar coordinates yield
    \begin{align*}
      &\int_{\substack{r<4\\R\geq8}}
      \ip\Xi^{-\varepsilon}
      \bigl(|\mathcal A_\zeta(\Xi)|+|\widehat{\zeta Q}(\Xi)|\bigr)\dd\Xi\\
      &\qquad\leq C_\zeta
      \int_0^4 r^2\dd r
      \int_8^\infty
      \bigl(R^{-1-\varepsilon}+R^{-5-\varepsilon}\bigr)\dd R\\
      &\qquad=
      \frac{64C_\zeta}{3}
      \left(
        \frac{8^{-\varepsilon}}{\varepsilon}
        +\frac{8^{-4-\varepsilon}}{4+\varepsilon}
      \right)
      \leq\frac{C_\zeta}{\varepsilon}.
    \end{align*}
    By \eqref{eq:localized-F-symmetry}, the region \(R<4\), \(r\geq8\) is identical after interchanging the frequency variables.

    Finally, consider the bounded region \(R,r<8\).  Since \(a_\zeta\in L^1(\R^3)\) and \(\zeta Q\in L^1(\R^6)\), their Fourier transforms are bounded.
    The definition of \(\mathcal A_\zeta\) therefore gives
    \[ |\mathcal A_\zeta(\Xi)|+|\widehat{\zeta Q}(\Xi)| \leq2\|\widehat a_\zeta\|_{L^\infty(\R^3)}^2 +\|\widehat{\zeta Q}\|_{L^\infty(\R^6)} \leq C_\zeta. \]
    Moreover, \(\ip\Xi^{-\varepsilon}\leq1\), and
    \[ \bigl|\{(\xi,\eta):R,r<8\}\bigr| =\left(\frac{4\pi}{3}8^3\right)^2. \]
    Hence
    \[
      \int_{\substack{R<8\\r<8}}
      \ip\Xi^{-\varepsilon}
      \bigl(|\mathcal A_\zeta(\Xi)|+|\widehat{\zeta Q}(\Xi)|\bigr)\dd\Xi
      \leq C_\zeta.
    \]

    Adding the four bounds in the displayed covering proves \eqref{eq:source-complement-error}.  Thus the complement contributes at most a simple pole.

    \emph{Step 5. Exact evaluation of the double pole.}
    The remaining task is to prove \eqref{eq:source-leading-pole}.
    Indeed, the global decomposition and Steps~1--4 give
    \[ \left| \|\res(\zeta\Gamma_0)\|_{\B^{2-\varepsilon}} -\int_{\Omega_4}\ip\Xi^{-\varepsilon}|\mathcal H(\Xi)|\dd\Xi \right| \leq\frac{C_\zeta}{\varepsilon}. \]
    Hence \eqref{eq:source-leading-pole} implies \eqref{eq:source-residue}.  It remains to evaluate the two cone integrals.
    On $D_4^+$, write
    \[ \ip\Xi^{-\varepsilon} =R^{-\varepsilon} \left(1+\frac{r^2+1}{R^2}\right)^{-\varepsilon/2}. \]
    The mean-value theorem applied to $s\mapsto(1+s)^{-\varepsilon/2}$ gives
    \[ 0\leq R^{-\varepsilon}-\ip\Xi^{-\varepsilon} \leq\frac{\varepsilon}{2}R^{-\varepsilon} \left(\frac{r^2+1}{R^2}\right). \]
    Since
    \[ \int_{\mathbb S^2}\int_{\mathbb S^2}|\omega\cdot\nu| \dd\omega\dd\nu=8\pi^2, \]
    the resulting weight error satisfies
    \begin{align*}
      &\int_{D_4^+}
      \bigl(R^{-\varepsilon}-\ip\Xi^{-\varepsilon}\bigr)
      |\mathcal H(\Xi)|\dd\Xi                                           \\
      &\quad\leq64\pi\frac{\varepsilon}{2}
      \int_{16}^\infty R^{-3-\varepsilon}
      \int_4^{R/4}\left(r+\frac1r\right)\dd r\dd R             \\
      &\quad\leq64\pi\left[
      \frac{4^{-2-2\varepsilon}}4
      +\frac{\varepsilon4^{-4-2\varepsilon}}
      {2(2+\varepsilon)^2}
      \right]
      \leq C.
    \end{align*}
    The homogeneous radial integral is exactly
    \begin{align*}
      \int_{16}^\infty R^{-1-\varepsilon}
      \int_4^{R/4}\frac{\dd r}{r}\dd R
       & =\int_{16}^\infty R^{-1-\varepsilon}
      \log\frac{R}{16}\dd R                       \\
       & =\frac{4^{-2\varepsilon}}{\varepsilon^2}.
    \end{align*}
    Hence
    \[ \int_{D_4^+}R^{-\varepsilon}|\mathcal H(\Xi)|\dd\Xi =\frac{64\pi\,4^{-2\varepsilon}}{\varepsilon^2}. \]
    The same two formulas hold on $D_4^-$, with $R$ and $r$ interchanged.  Therefore
    \[ \left| \int_{\Omega_4}\ip\Xi^{-\varepsilon}|\mathcal H(\Xi)|\dd\Xi -\frac{128\pi\,4^{-2\varepsilon}}{\varepsilon^2} \right|\leq C. \]
    Finally,
    \[ 1-4^{-2\varepsilon} =2\varepsilon\log4\int_0^1 4^{-2t\varepsilon}\dd t. \]
    Consequently,
    \[ \left| \frac{128\pi\,4^{-2\varepsilon}}{\varepsilon^2} -\frac{128\pi}{\varepsilon^2} \right| \leq\frac{256\pi\log4}{\varepsilon}. \]
    This proves \eqref{eq:source-leading-pole} and hence \eqref{eq:source-residue}.
  \end{proof}

\subsection{The local conjugated equation}

This subsection identifies the critical coefficient in the local conjugated equation. We compute $|\nabla F|^2$ and the distributional Laplacian of $P$, verify that no distribution supported at the triple coalescence is omitted, and combine their nonconstant parts into $-(Z/4)\Gamma_0$. The resulting identity \eqref{eq:local-conjugated} separates this source from the drift and all lower-order terms. It provides the exact decomposition used in the next two subsections to preserve the double-pole term and bound every remainder by $C/\varepsilon$.

Choose $\vartheta\in C_c^\infty(\R^6)$ equal to one near $(0,0)$ and supported where all radial cut-offs equal one. On this support,
\begin{equation*} F=-\frac Z2(|x|+|y|)+\frac14|x-y|, \qquad P=\frac{(2-\pi)Z}{12\pi}(x\cdot y) \log(|x|^2+|y|^2). \end{equation*}
The gradients of the two-particle factor are
\begin{equation*} \nabla_xF=-\frac Z2\frac{x}{|x|} +\frac14\frac{x-y}{|x-y|}, \qquad \nabla_yF=-\frac Z2\frac{y}{|y|} -\frac14\frac{x-y}{|x-y|}. \end{equation*}
Squaring both identities and using that each displayed unit vector has length one almost everywhere gives
\begin{align*}
  |\nabla_xF|^2
   & =\frac{Z^2}{4}+\frac1{16}
  -\frac Z4\frac{x}{|x|}\cdot\frac{x-y}{|x-y|},   \\
  |\nabla_yF|^2
   & =\frac{Z^2}{4}+\frac1{16}
  +\frac Z4\frac{y}{|y|}\cdot\frac{x-y}{|x-y|}.
\end{align*}
Hence
\begin{equation*} |\nabla F|^2 =\frac{Z^2}{2}+\frac18-\frac Z4\Gamma. \end{equation*}

It remains to check the numerical coefficient supplied by $P$.  Direct differentiation gives
\[ \nabla_x\log(|x|^2+|y|^2)=\frac{2x}{|x|^2+|y|^2}, \qquad \Delta_x\log(|x|^2+|y|^2) =\frac6{|x|^2+|y|^2} -\frac{4|x|^2}{(|x|^2+|y|^2)^2}, \]
the product rule gives
\[ \Delta_x\bigl[(x\cdot y)\log(|x|^2+|y|^2)\bigr] =\frac{10x\cdot y}{|x|^2+|y|^2} -\frac{4|x|^2(x\cdot y)}{(|x|^2+|y|^2)^2}. \]
Adding the formula with $x$ and $y$ interchanged yields
\begin{equation*} (\Delta_x+\Delta_y) \bigl[(x\cdot y)\log(|x|^2+|y|^2)\bigr] =16\frac{x\cdot y}{|x|^2+|y|^2}=16Q. \end{equation*}
This identity also holds in $\mathcal D'(\R^6)$.  To verify that no distribution supported at the origin is missed, set $h(X)=(x\cdot y)\log|X|^2$, where $X=(x,y)$.  For $0<|X|<1$ and $|\alpha|\leq2$, direct differentiation gives
\[ |\partial^\alpha h(X)| \leq C_\alpha |X|^{2-|\alpha|} \bigl(1+|\log|X||\bigr). \]
In particular,
\[ \int_{|X|<1}|D^2h(X)|\dd X \leq C\int_0^1 r^5(1+|\log r|)\dd r<\infty. \]
Moreover, when the distributional derivatives are computed on $\{|X|>\delta\}$, the first and second integrations by parts produce boundary terms bounded, respectively, by $C\delta^7(1+|\log\delta|)$ and $C\delta^6(1+|\log\delta|)$; both tend to zero as $\delta\downarrow0$.
Thus $h\in W^{2,1}_{\mathrm{loc}}(\R^6)$, its weak second derivatives agree with the displayed classical derivatives away from the origin, and no Dirac mass occurs.
Consequently, $\Delta P=4(2-\pi)ZQ/(3\pi)$.  Combining this identity with $|\nabla F|^2$ gives
\begin{equation}\label{eq:critical-group-expanded} \Delta P+|\nabla F|^2 =\frac{Z^2}{2}+\frac18-\frac Z4\Gamma_0. \end{equation}

Applying the conjugation identity used in \eqref{eq:phi3-equation}, using $\Delta F=V$ locally, and inserting \eqref{eq:critical-group-expanded} gives
\begin{equation}\label{eq:local-conjugated} (-\Delta+1)u =2\nabla(F+P)\cdot\nabla u-\frac Z4\Gamma_0u+S_{\rm loc}u, \end{equation}
where
\begin{equation}\label{eq:local-zero-order} S_{\rm loc}=E+1+\frac{Z^2}{2}+\frac18 +2\nabla F\cdot\nabla P+|\nabla P|^2. \end{equation}
The critical term therefore has the fixed sign $-Z\Gamma_0u/4$.
The combination in \eqref{eq:critical-group-expanded}, rather than $\Gamma$ and $Q$ separately, must be analyzed before Fourier absolute values are taken.

\subsection{Freezing the quotient at the triple coalescence}

The purpose of this subsection is to replace $u$ in the critical source by its value $u(0,0)$ without altering the coefficient of the double pole. We first prove that one frequency derivative of $\widehat{\vartheta\Gamma_0}$ has weighted $L^1$ norm bounded by $C/\varepsilon$. The mean-value formula converts this derivative estimate into a translation-difference bound. Fourier inversion then represents $\vartheta\Gamma_0(u-u(0,0))$ by precisely this difference, yielding \eqref{eq:frozen-coefficient}. Thus the variable part of $u$ contributes only a simple pole.

\begin{lemma}[Derivative and translation-difference estimates]
  There is $C_{\rm der}<\infty$ such that, for $0<\varepsilon\leq1/2$ and $\Theta\in\R^6$,
  \begin{align}
    \int_{\R^6}\ip\Xi^{-\varepsilon}
    |\nabla\widehat{\vartheta\Gamma_0}(\Xi)|\dd\Xi
     & \leq\frac{C_{\rm der}}{\varepsilon},
     \label{eq:critical-source-derivative} \\
    \int_{\R^6}\ip\Xi^{-\varepsilon}
    |\widehat{\vartheta\Gamma_0}(\Xi-\Theta)
      -\widehat{\vartheta\Gamma_0}(\Xi)|\dd\Xi
     & \leq\frac{2^{1/4}C_{\rm der}}{\varepsilon}
    |\Theta|\ip\Theta^\varepsilon.                       \label{eq:critical-source-difference}
  \end{align}
\end{lemma}

\begin{proof}
  Apply Lemma~\ref{lem:localized-angular-source} with $\rho=\vartheta$.
  Differentiating \eqref{eq:localized-source-representation}, using \eqref{eq:weighted-convolution} for the convolution term, and then applying \eqref{eq:localized-F-derivative} and \eqref{eq:localized-Q-gradient}, we obtain
  \begin{align*}
    &\int_{\R^6}\ip\Xi^{-\varepsilon}
      |\nabla\widehat{\vartheta\Gamma_0}(\Xi)|\dd\Xi\\
    &\quad\leq
      \frac{2^{1/4}C_\vartheta}{\varepsilon}
      \|\ip{\cdot}^{1/2}k_\vartheta\|_{L^1(\R^6)}
      +C_\vartheta\int_{\R^6}\ip\Xi^{-7}\dd\Xi
      \leq\frac{C_{\rm der}}{\varepsilon}.
  \end{align*}
  The two moments are finite because $k_\vartheta$ is Schwartz and $\ip{\cdot}^{-7}\in L^1(\R^6)$.  This proves \eqref{eq:critical-source-derivative}.  Applying \eqref{eq:weighted-translation} to $G=\widehat{\vartheta\Gamma_0}$ and using $2^{\varepsilon/2}\leq2^{1/4}$ proves \eqref{eq:critical-source-difference}.
\end{proof}

\begin{lemma}[Freezing at the triple coalescence]\label{lem:frozen-coefficient}
  For $f\in\B^{3/2}(\R^6)$ and $0<\varepsilon\leq1/2$,
  \begin{equation}\label{eq:frozen-coefficient} \|\res(\vartheta\Gamma_0(f-f(0)))\|_{\B^{2-\varepsilon}} \leq\frac{2^{1/4}C_{\rm der}}{(2\pi)^3\varepsilon} \|f\|_{\B^{3/2}}. \end{equation}
\end{lemma}

\begin{proof}
  Since $f\in\B^{3/2}\subset\B^0$, Fourier inversion converges absolutely and
  \[ f(0)=(2\pi)^{-3}\int_{\R^6}\widehat f(\Theta)\dd\Theta. \]
  The unitary product formula therefore gives the exact kernel-difference identity
  \[ \widehat{\vartheta\Gamma_0(f-f(0))}(\Xi) =(2\pi)^{-3}\int_{\R^6} [\widehat{\vartheta\Gamma_0}(\Xi-\Theta) -\widehat{\vartheta\Gamma_0}(\Xi)] \widehat f(\Theta)\dd\Theta. \]
  Applying \eqref{eq:resolvent-lift-expanded} with $p=1$ and $2\beta=\varepsilon$, followed by Tonelli's theorem and \eqref{eq:critical-source-difference}, gives
  \begin{align*}
    \|\res(\vartheta\Gamma_0(f-f(0)))\|_{\B^{2-\varepsilon}}
     & \leq\frac{2^{1/4}C_{\rm der}}{(2\pi)^3\varepsilon}
    \int_{\R^6}|\Theta|\ip\Theta^\varepsilon
    |\widehat f(\Theta)|\dd\Theta                 \\
     & \leq\frac{2^{1/4}C_{\rm der}}{(2\pi)^3\varepsilon}
    \|f\|_{\B^{3/2}},
  \end{align*}
  because $|\Theta|\ip\Theta^\varepsilon\leq\ip\Theta^{3/2}$ for $\varepsilon\leq1/2$.
\end{proof}

\subsection{All remaining terms have at most one pole}

This subsection proves the two simple-pole estimates in Lemma~\ref{lem:one-pole-remainder}. In the local region, the coefficient identity \eqref{eq:local-zero-order} leaves at most one angular block in every remainder. In the exterior region, a smooth partition of each product associated with two distinct collision coordinates places one angular block away from its singularity, so again only one critical scale remains. The structured Wiener and low-rank multiplier estimates, followed by the resolvent identity, give \eqref{eq:local-remainder} and \eqref{eq:global-remainder}. Together with the freezing estimate, these bounds show that every term other than the frozen source is $O(\varepsilon^{-1})$ in the final norm comparison.

\begin{lemma}[Local and exterior remainders]\label{lem:one-pole-remainder}
  There are finite constants $C_{\rm loc}$ and $C_{\rm ext}$, independent of $0<\varepsilon\leq1/2$, such that
  \begin{align}
    \|\res(2\vartheta\nabla(F+P)\cdot\nabla u+
    \vartheta S_{\rm loc}u+[-\Delta,\vartheta]u)\|_{\B^{2-\varepsilon}}
     & \leq\frac{C_{\rm loc}}{\varepsilon}\|u\|_{\B^{3/2}},
    \label{eq:local-remainder}                              \\
    \|(1-\vartheta)u\|_{\B^{2-\varepsilon}}
     & \leq\frac{C_{\rm ext}}{\varepsilon}\|u\|_{\B^{3/2}}.
    \label{eq:global-remainder}
  \end{align}
\end{lemma}

\begin{proof}
  \emph{Local terms.}
  Multiplication by a fixed $\theta\in C_c^\infty(\R^6)$ is bounded on $\FL_{-\varepsilon}^1$, uniformly for $0<\varepsilon\leq1/2$.  Indeed, the unitary product formula and \eqref{eq:weighted-convolution}, applied with $k=(2\pi)^{-3}\widehat\theta$, give
  \[ \|\theta f\|_{\FL_{-\varepsilon}^1} \leq \frac{2^{1/4}}{(2\pi)^3} \|\ip{\cdot}^{1/2}\widehat\theta\|_{L^1(\R^6)} \|f\|_{\FL_{-\varepsilon}^1}. \]
  Equations \eqref{eq:q7} and \eqref{eq:local-zero-order} give the exact coefficient classification
  \[
  \begin{array}{c|c}
    \text{term} & \text{coefficient type}\\ \hline
    \nabla F\cdot\nabla u & \text{one angular block}\\
    \nabla P\cdot\nabla u & \text{Wiener}\\
    S_{\rm loc}u & \text{constant, Wiener, or angular--Wiener}
  \end{array}
  \]
  (finite sums are understood). Thus every local term contains at most one angular block. If $b$ is an angular block and $c$ is a Wiener coefficient, Lemma~\ref{lem:structured-wiener} and \eqref{eq:m-operator-expanded}, applied in this order, give
  \[ \|cf\|_{\FL_{-\varepsilon}^1} \leq C_c\|f\|_{\B^0}, \qquad \|bcf\|_{\FL_{-\varepsilon}^1} \leq\frac{C_{b,c}}{\varepsilon}\|f\|_{\B^0}. \]
  Finally,
  \begin{equation*} [-\Delta,\vartheta]u =-(\Delta\vartheta)u-2\nabla\vartheta\cdot\nabla u \end{equation*}
  has smooth compactly supported coefficients.  Since \(\|\partial_j u\|_{\B^0}\leq\|u\|_{\B^{3/2}}\), the preceding bounds give
  \[ \|2\vartheta\nabla(F+P)\cdot\nabla u+\vartheta S_{\rm loc}u +[-\Delta,\vartheta]u\|_{\FL_{-\varepsilon}^1} \leq\frac{C_{\rm loc}}{\varepsilon}\|u\|_{\B^{3/2}}. \]
  Applying \eqref{eq:resolvent-lift-expanded} with $p=1$ and $2\beta=\varepsilon$ proves \eqref{eq:local-remainder}.

  \emph{Exterior term.}
  Set $u_{\rm ext}=(1-\vartheta)u$.  Multiplying the global equation \eqref{eq:phi3-equation} by $1-\vartheta$ and using $(-\Delta+1)((1-\vartheta)u)=(1-\vartheta)(-\Delta+1)u +[-\Delta,1-\vartheta]u$ gives the exact identity
  \begin{align}
    (-\Delta+1)u_{\rm ext}
     &=(1-\vartheta)\Bigl\{2\nabla(F+P)\cdot\nabla u \notag\\
     &\qquad+\bigl[\Delta(F-F_2)+\Delta P
       +|\nabla(F+P)|^2+E+1\bigr]u\Bigr\}
       +[-\Delta,1-\vartheta]u.
       \label{eq:exterior-equation-expanded}
  \end{align}

  We first make the collision geometry precise.  The three collision sets are
  \[ \Sigma_x=\{x=0\},\qquad \Sigma_y=\{y=0\},\qquad \Sigma_{xy}=\{x=y\}. \]
  Their pairwise intersections equal $\{(0,0)\}$.  Choose $c_0>0$ such that $1-\vartheta=0$ on $\{|x|^2+|y|^2<c_0^2\}$.  Let $L_1X$ and $L_2X$ be two distinct collision coordinates chosen from $x$, $y$, and $x-y$, where $X=(x,y)$, and set $\ell_j=|L_jX|$.  Each map $X\mapsto(L_1X,L_2X)$ is invertible.  Since there are only three choices of pairs, a single $\kappa>0$ satisfies
  \[ \ell_1^2+\ell_2^2\geq\kappa(|x|^2+|y|^2). \]
  Hence, on $\supp(1-\vartheta)$,
  \[ \ell_1^2+\ell_2^2\geq \kappa c_0^2. \]
  Fix two components $b_\alpha$ and $b_\beta$ of the cut-off angular field and consider the product $b_\alpha(L_1X)b_\beta(L_2X)$.  Set
  \[ K_{12}^{\rm ext} :=\supp\bigl((1-\vartheta)b_\alpha(L_1X)b_\beta(L_2X)\bigr). \]
  Both $\ell_1$ and $\ell_2$ are bounded on the support of the product.
  Since $X\mapsto(L_1X,L_2X)$ is invertible, $K_{12}^{\rm ext}$ is compact.
  Put $r_{\rm sep}=c_0\sqrt{\kappa/2}$.  Choose $h\in C^\infty([0,\infty))$ with
  \[ 0\leq h\leq1,\qquad h(t)=0\iff t\leq r_{\rm sep}/2,\qquad h(t)=1\quad\text{for }t\geq r_{\rm sep}. \]
  Also choose $\omega_0\in C_c^\infty(\R^6)$ equal to one on $\{|x|^2+|y|^2\leq c_0^2/4\}$ and equal to zero on a neighborhood of $K_{12}^{\rm ext}$.  Define on all of $\R^6$
  \[ \pi_j(x,y)= \frac{h(\ell_j)}{h(\ell_1)+h(\ell_2)+\omega_0(x,y)}, \qquad j=1,2. \]
  The denominator is positive because
  \[
  \begin{aligned}
    h(\ell_1)+h(\ell_2)=0
    &\Longrightarrow \ell_1,\ell_2\leq r_{\rm sep}/2\\
    &\Longrightarrow
    \kappa(|x|^2+|y|^2)\leq\ell_1^2+\ell_2^2
    \leq\frac{r_{\rm sep}^2}{2}=\frac{\kappa c_0^2}{4}\\
    &\Longrightarrow \omega_0(x,y)=1.
  \end{aligned}
  \]
  Moreover, $h(|z|)$ is smooth at $z=0$ because it vanishes near the origin.  Thus $\pi_1$ and $\pi_2$ are globally smooth.
  On $K_{12}^{\rm ext}$,
  \[ \omega_0=0,\qquad \max\{\ell_1,\ell_2\} \geq\left(\frac{\ell_1^2+\ell_2^2}{2}\right)^{1/2} \geq r_{\rm sep},\qquad \pi_1+\pi_2=1. \]
  Moreover,
  \[ \pi_j(X)\neq0\quad\Longrightarrow\quad \ell_j>r_{\rm sep}/2, \]
  so the $j$th angular block is smooth on the $j$th partitioned term.

  Since $\pi_1+\pi_2=1$ on $K_{12}^{\rm ext}$ and the left-hand side below vanishes outside this set, inserting the partition gives the exact decomposition
  \[
  \begin{aligned}
    &(1-\vartheta)b_\alpha(L_1X)b_\beta(L_2X)\\
    &\qquad=(1-\vartheta)\pi_1b_\alpha(L_1X)b_\beta(L_2X)
      +(1-\vartheta)\pi_2b_\alpha(L_1X)b_\beta(L_2X).
  \end{aligned}
  \]
  Choose one function $\rho\in C_c^\infty(\R^3)$ such that $\rho=1$ on $\supp b$.  Thus $\rho b_\gamma=b_\gamma$ for every component $b_\gamma$ of $b$.  After the change of variables $(z,w)=(L_1X,L_2X)$, the two terms become
  \[ g_1(z,w)b_\beta(w),\qquad g_2(z,w)b_\alpha(z), \]
  where
  \[
  \begin{aligned}
    g_1(z,w)&=(1-\vartheta)\pi_1b_\alpha(z)\rho(w),\\
    g_2(z,w)&=(1-\vartheta)\pi_2\rho(z)b_\beta(w).
  \end{aligned}
  \]
  Here $1-\vartheta$ and $\pi_j$ are understood after composition with the inverse coordinate map.  Thus $g_1,g_2\in C_c^\infty(\R^6)$.  Indeed, $b_\alpha(z)\rho(w)$ and $\rho(z)b_\beta(w)$ give compact support, while $\pi_1\neq0$ implies $|z|>r_{\rm sep}/2$ and $\pi_2\neq0$ implies $|w|>r_{\rm sep}/2$.  Hence the angular factor absorbed into each $g_j$ is smooth on its support.

  For $g\in C_c^\infty(\R^6)$, multiplication occurs only in the three-dimensional $w$ variable.  Hence the unitary product formula carries the factor $(2\pi)^{-3/2}$ and gives
  \[ \widehat{g b}(\xi,\eta) =(2\pi)^{-3/2}\int_{\R^3}\widehat g(\xi,\eta-\nu) \widehat b(\nu)\dd\nu. \]
  Young's inequality in $\eta$, followed by integration in $\xi$, gives, for $1<\sigma\leq2$,
  \begin{align*}
    \|\widehat{g b}\|_{L^\sigma(\R^6)}
     & \leq(2\pi)^{-3/2}
    \left\|\|\widehat g(\xi,\cdot)\|_{L^1(\R^3)}
    \right\|_{L^\sigma(\R^3_\xi)}
    \|\widehat b\|_{L^\sigma(\R^3)}.
  \end{align*}
  Because $\widehat g$ is Schwartz, interpolation between $\sigma=1$ and $\sigma=2$ gives
  \[ \sup_{1\leq\sigma\leq2} \left\|\|\widehat g(\xi,\cdot)\|_{L^1(\R^3)} \right\|_{L^\sigma(\R^3_\xi)}<\infty. \]
  Taking $\sigma=r_\varepsilon$ and using \eqref{eq:b-r-exact-expanded} yields the explicit one-pole bound
  \[ \|\widehat{g b}\|_{L^{r_\varepsilon}(\R^6)} \leq\frac{C_g}{\varepsilon}, \qquad 0<\varepsilon\leq\frac12. \]
  In the diagonal terms of $|\nabla F|^2$, the square of $z/|z|$ is one; terms containing derivatives of the fixed radial cut-offs are smooth and compactly supported.  They are therefore noncritical.

  The coefficients in \eqref{eq:exterior-equation-expanded} now have the following precise classification.
  \[
  \begin{array}{c|c}
    \text{coefficient} & \text{Fourier structure}\\ \hline
    \nabla F & \text{one angular block}\\
    \nabla P,\ \Delta(F-F_2),\ |\nabla P|^2
      & \text{Wiener}\\
    (1-\vartheta)\Delta P
      & C_c^\infty(\R^6)\\
    2\nabla F\cdot\nabla P
      & \text{one angular block times a Wiener factor}\\
    |\nabla F|^2
      & \text{constant, smooth, or a partitioned one-angular term}.
  \end{array}
  \]
  Here the last row uses the preceding partition only for products associated with distinct collision coordinates; diagonal angular squares equal one away from the harmless cut-off derivatives.  Multiplication by $1-\vartheta$ preserves this classification because
  \[ (1-\vartheta)b=b-\vartheta b, \qquad \widehat{(1-\vartheta)c} =\widehat c-(2\pi)^{-3}\widehat\vartheta*\widehat c, \quad \widehat c\in L^1. \]
  For a three-dimensional Wiener block $d_J(LX)$, whose full Fourier transform contains spectator delta distributions, the corresponding operator identity is
  \[ (1-\vartheta)d_J(LX)f=d_J(LX)f-\vartheta\,[d_J(LX)f]. \]
  Lemma~\ref{lem:structured-wiener}, with the three-dimensional active Fourier transform $\widehat {d_J}\in L^1(\R^3)$, and the smooth-cut-off estimate \eqref{eq:weighted-convolution} give
  \[
  \begin{aligned}
    \|d_J(LX)f\|_{\FL_{-\varepsilon}^1}
      &\leq C_{d_J}\|f\|_{\B^0},\\
    \|\vartheta[d_J(LX)f]\|_{\FL_{-\varepsilon}^1}
      &\leq C_\vartheta\|d_J(LX)f\|_{\FL_{-\varepsilon}^1}
      \leq C_\vartheta C_{d_J}\|f\|_{\B^0}.
  \end{aligned}
  \]
  Hence
  \[ \|(1-\vartheta)d_J(LX)f\|_{\FL_{-\varepsilon}^1} \leq(1+C_\vartheta)C_{d_J}\|f\|_{\B^0}, \]
  without treating the spectator deltas as ordinary $L^1(\R^6)$ functions.

  The remaining commutator is
  \[ [-\Delta,1-\vartheta]u=(\Delta\vartheta)u +2\nabla\vartheta\cdot\nabla u \]
  and has smooth compactly supported coefficients. Lemmas~\ref{lem:structured-wiener} and \ref{lem:multiplier}, together with $\|\nabla u\|_{\B^0}\leq\|u\|_{\B^{3/2}}$, therefore give the source estimate
  \[ \|(-\Delta+1)u_{\rm ext}\|_{\FL_{-\varepsilon}^1} \leq\frac{C_{\rm ext}}{\varepsilon}\|u\|_{\B^{3/2}}. \]
  Applying \eqref{eq:resolvent-lift-expanded} with $p=1$ and $2\beta=\varepsilon$ gives
  \[ \|u_{\rm ext}\|_{\B^{2-\varepsilon}} \leq\frac{C_{\rm ext}}{\varepsilon}\|u\|_{\B^{3/2}}. \]
  This is \eqref{eq:global-remainder}.
\end{proof}

\begin{proof}[Proof of Theorem~\ref{thm:main}(iv)]
  We first write $u$ as the sum of the frozen critical source, its variable-coefficient error, the remaining local terms, and the exterior part.  The reverse triangle inequality then reduces \eqref{eq:sharp-main} to four explicit error bounds.  Proposition~\ref{prop:source-residue} controls the first, Lemma~\ref{lem:frozen-coefficient} the second, and Lemma~\ref{lem:one-pole-remainder} the last two.

  Applying Theorem~\ref{thm:main}(i) with $s=3/2$ gives $u\in\B^{3/2}(\R^6)$ independently of $\varepsilon$.  Multiplying \eqref{eq:local-conjugated} by $\vartheta$ and using
  \[ \vartheta\Gamma_0u =u(0,0)\vartheta\Gamma_0 +\vartheta\Gamma_0(u-u(0,0)) \]
  and then adding $(1-\vartheta)u$ gives the exact global decomposition
  \begin{align*}
    u={}&-\frac Z4u(0,0)\res(\vartheta\Gamma_0)
       -\frac Z4\res\bigl(\vartheta\Gamma_0
         (u-u(0,0))\bigr)                                      \\
      &+\res\bigl(2\vartheta\nabla(F+P)\cdot\nabla u
         +\vartheta S_{\rm loc}u+[-\Delta,\vartheta]u\bigr)
       +(1-\vartheta)u .
  \end{align*}
  The first term is the only one with two independent critical frequency scales.  Since $|\|f+g\|-\|f\||\leq\|g\|$ in every normed space, and since the norm of the first term is
  \[ \frac{Z|u(0,0)|}{4} \|\res(\vartheta\Gamma_0)\|_{\B^{2-\varepsilon}}, \]
  the preceding identity yields the following complete error decomposition.
  \begin{align*}
    &\left|
      \|u\|_{\B^{2-\varepsilon}}
      -\frac{32\pi Z|u(0,0)|}{\varepsilon^2}
    \right|                                                     \\
    &\quad\leq
      \frac{Z|u(0,0)|}{4}
      \left|
        \|\res(\vartheta\Gamma_0)\|_{\B^{2-\varepsilon}}
        -\frac{128\pi}{\varepsilon^2}
      \right|                                                   \\
    &\qquad+
      \frac Z4
      \|\res(\vartheta\Gamma_0(u-u(0,0)))
      \|_{\B^{2-\varepsilon}}                                  \\
    &\qquad+
      \|\res(2\vartheta\nabla(F+P)\cdot\nabla u
        +\vartheta S_{\rm loc}u+[-\Delta,\vartheta]u)
      \|_{\B^{2-\varepsilon}}                                  \\
    &\qquad+
      \|(1-\vartheta)u\|_{\B^{2-\varepsilon}} .
  \end{align*}

  The four terms on the right are controlled, in order, by \eqref{eq:source-residue}, multiplied by $Z|u(0,0)|/4$; \eqref{eq:frozen-coefficient}, with $f=u$ and multiplied by $Z/4$; \eqref{eq:local-remainder}; and \eqref{eq:global-remainder}.
  Substitution into the error decomposition gives
  \begin{align*}
    &\left|
      \|u\|_{\B^{2-\varepsilon}}
      -\frac{32\pi Z|u(0,0)|}{\varepsilon^2}
    \right|                                                     \\
    &\quad\leq\frac1\varepsilon
      \left[
        \frac{Z|u(0,0)|C_\vartheta}{4}
        +\left(
          \frac{Z2^{1/4}C_{\rm der}}{4(2\pi)^3}
          +C_{\rm loc}+C_{\rm ext}
        \right)\|u\|_{\B^{3/2}}
      \right].
  \end{align*}
  The bracket is finite and independent of $\varepsilon$; absorbing it into $C_{\rm sharp}$ proves \eqref{eq:sharp-main}.

  For $Z>1$, Zhislin's binding theorem provides a scalar two-electron ground state below the essential spectrum; see \cite{Zhislin1960} and \cite[Corollary~11.10]{Simon2019KatoPart2}. The Harnack principle gives a continuous strictly positive representative \cite[Theorems~C.1.1 and C.1.3]{Simon1982}. Both cut-off Jastrow factors vanish at $(0,0)$, so $u(0,0)=\psi(0,0)>0$. Hence the double pole is attained by an eigenfunction of the unperturbed Coulomb Hamiltonian, and the exponent two in \eqref{eq:upper} is sharp.
\end{proof}

\section{Concluding remarks}

This paper determines the Barron regularity gained by removing the explicit two- and three-particle Coulomb singularities from an electronic eigenfunction. For both successive quotients, we prove membership in $\B^s$ for every $s<2$ and obtain the quantitative estimate $\|u\|_{\B^{2-\varepsilon}}\leq M\varepsilon^{-2}\|u\|_{\B^1}$. We also prove that neither conclusion can be improved in its stated sense. No universal state-independent multiplicative factor in the class specified in Theorem~\ref{thm:main}(iii) can place all Coulombic quotients in $\B^2$, and an unperturbed two-electron eigenfunction satisfies the exact leading asymptotic
\[
  \|\phi_3\|_{\B^{2-\varepsilon}(\R^6)}
  =32\pi Z|\phi_3(0,0)|\varepsilon^{-2}+\mathcal{O}(\varepsilon^{-1})
\]
whenever $\phi_3(0,0)\neq0$.

The analysis shows that the endpoint growth is governed by the number of independent collision-frequency scales rather than by the ambient dimension $3N$. The low-rank Fourier--Lebesgue formulation isolates these three- and six-dimensional structures, while the residue computation proves that the quadratic divergence is an intrinsic feature of the three-particle coalescence and not a loss introduced by the multiplier upper bound. This connection between local Coulomb geometry and global weighted Fourier integrability also provides a regularity input for studying high-dimensional approximation of the factored wave function. The results proved here require no permutation symmetry.

Two directions arise naturally from this work. For fermionic wave functions, mixed regularity is particularly important because antisymmetry yields additional regularity across electron variables and can improve high-dimensional approximation rates; see~\cite{Yserentant2010,Meng2023}. Corresponding results on mixed spectral Barron regularity for antisymmetric Coulombic wave functions will be reported in a forthcoming work. The low-regularity $L^2$- and $L^\infty$-approximation estimates for deep neural networks in~\cite{LiaoMingYu2025}, together with the generalization framework for Schr\"odinger eigenvalue solvers in \cite{GuoMingYu2026}, provide a basis for deriving quantitative approximation and generalization error bounds for neural-network computation of Coulombic wave functions.

\bibliographystyle{amsplain}
\bibliography{references}

\end{document}